%% file: FranceschiStefani.tex
\documentclass[a4paper,12pt,reqno]{amsart}
\usepackage[T1]{fontenc}
\usepackage[utf8]{inputenc} 
\usepackage{lmodern}
\usepackage[english]{babel}
\usepackage[normalem]{ulem}
\usepackage[%
	left=2.5cm,       
	right=2.5cm,      
	top=3.5cm,        
	bottom=3.5cm,     
	heightrounded,    
	bindingoffset=0mm 
]{geometry}

\usepackage{amssymb}
\usepackage{mathtools}

\theoremstyle{plain}
	\newtheorem{theorem}{Theorem}[section]
	\newtheorem{lemma}[theorem]{Lemma}
	\newtheorem{proposition}[theorem]{Proposition}

\theoremstyle{definition}
	\newtheorem{definition}[theorem]{Definition}
	\newtheorem{example}[theorem]{Example}
	\newtheorem{remark}[theorem]{Remark}

\numberwithin{equation}{section}

\usepackage{tikz}

\usepackage{pgfplots}
	\pgfplotsset{compat=1.11}

\usepackage{graphicx}
\usepackage{xcolor}

\usepackage{caption}
\usepackage{url}

\usepackage{enumerate}

\newcommand{\N}{\mathbb{N}}
\newcommand{\R}{\mathbb{R}}

\newcommand{\eps}{\varepsilon}
\newcommand{\de}{\partial}

\usepackage{calrsfs}
\newcommand{\A}{\mathcal{A}}

\newcommand{\M}{\mathcal{M}}
\newcommand{\leb}{\mathcal{L}}
\newcommand{\haus}{\mathcal{H}}

\newcommand{\mres}{
	\mathbin{\vrule height 1.6ex depth 0pt width 0.13ex
	         \vrule height 0.13ex depth 0pt width 1.3ex}}

\renewcommand{\phi}{\varphi}
\renewcommand{\P}{\mathcal{P}}
\renewcommand{\S}{\mathcal{S}}
\renewcommand{\theta}{\vartheta}
\renewcommand{\rho}{\varrho}

\DeclareMathOperator{\Lip}{Lip}
\DeclareMathOperator{\sgn}{sgn}
\DeclareMathOperator{\co}{co}

\DeclareMathOperator{\ess}{ess}
\DeclareMathOperator{\dom}{dom}

\DeclarePairedDelimiter{\scalar}{<}{>}                                     
\DeclarePairedDelimiter{\set}{\{}{\}}
\DeclarePairedDelimiter{\abs}{|}{|}

\mathchardef\ordinarycolon\mathcode`\:
\mathcode`\:=\string"8000
\begingroup \catcode`\:=\active
  \gdef:{\mathrel{\mathop\ordinarycolon}}
\endgroup

\usepackage{bookmark}
\usepackage{hyperref}

\usepackage[initials]{amsrefs}

\begin{document}

\title{Symmetric double bubbles in the Grushin plane}

\thanks{\textit{Acknowledgments}. The authors would like to thank Prof.~Roberto Monti from the University of Padova for several suggestions on the problem.
The authors would also like to thank Prof.~G.~P.~Leonardi from the University of Modena and Reggio Emilia for the useful discussions on the subject.
The first author also acknowledges the support of the project ANR-15-CE40-0018, of 
 the G.N.A.M.P.A.\ project \textit{Problemi nonlocali e degeneri in $\R^n$} and of the INS2I project \textit{CINCIN}}

\author{Valentina Franceschi$^\flat$}
\address{$^\flat$FMJH \& IMO, B\^atiment 307, Facult\'e des Sciences d'Orsay, Universit\'e Paris Sud, Orsay}
\email{valentina.franceschi@math.u-psud.fr}

\author{Giorgio Stefani$^\sharp$}
\address{$^\sharp$Scuola Normale Superiore, Piazza Cavalieri 7, 56126 Pisa, Italy}
\email{giorgio.stefani@sns.it}

\date{\today}
\begin{abstract}
We address the {\em double bubble} problem for the anisotropic {\em Grushin perimeter} $P_\alpha$, $\alpha\geq 0$, and the Lebesgue measure in $\R^2$, in the case of two equal volumes. 
We assume that 
the \emph{contact interface} between the bubbles lies on either the vertical or the horizontal axis.  We first prove existence of minimizers via the direct method by symmetrization arguments and then characterize them in terms of the given area by first variation techniques. Even though no regularity theory is available in this setting, we prove that angles at which minimal boundaries intersect satisfy the standard \emph{$120$-degree rule} up to a suitable change of coordinates. While for $\alpha=0$ the Grushin perimeter reduces to the Euclidean one and both minimizers coincide with the symmetric double bubble found in~\cite{FABHZ93}, for $\alpha=1$ vertical interface minimizers have Grushin perimeter strictly greater than horizontal interface minimizers. As the latter ones are obtained by translating and dilating the Grushin isoperimetric set found in~\cite{MM04}, we conjecture that they solve the double bubble problem with no assumptions on the contact interface.
 \end{abstract}
%
%

\maketitle
\section{Introduction}
\label{sec:intro}

\subsection{General framework}

For a volume measure $\mathsf{V}$ and a perimeter measure $\mathsf{P}$ on an $n$-dimensional manifold $\mathbb{M}$, an \emph{$m$-bubble cluster} is a family of $m\ge2$ pairwise disjoint sets $\set*{E_i\subset\mathbb{M} : i=1,\dots,m}$ such that $\mathsf{P}(E_i)<+\infty$ and $\mathsf{V}(E_i)<+\infty$ for all $i=1,\dots,m$.
Given $v_i>0$ for $i=1,\dots,m$, a \emph{minimal bubble cluster problem} on $(\mathbb{M},\mathsf{P},\mathsf{V})$ consists in finding solutions of the minimization problem
\begin{equation}\label{eq:intro_bubble_cluster_problem}\tag{\textbf{P}}
\inf\set*{\P_\mathsf{P}(E) : E=\bigcup_{i=1}^m E_i\subset\mathbb{M},\ \mathsf{V}(E_i)=v_i},
\end{equation}
where $\set*{E_i : i=1,\dots,m}$ is an $m$-bubble cluster and $\P_\mathsf{P}$ is given by 
\begin{equation*}
\P_\mathsf{P}(E)=\frac{1}{2}\,\mathsf{P}(\mathbb{M}\setminus E)+\frac{1}{2}\sum_{i=1}^m\mathsf{P}(E_i).
\end{equation*}


For $m=1$, \eqref{eq:intro_bubble_cluster_problem} is the \emph{isoperimetric problem}. When $\mathbb{M}$ is either the Euclidean space, the $n$-dimensional sphere $\mathbb S^n$ or the hyperbolic space, endowed with the Riemannian perimeter and volume, minimizers are known to be metric balls, see~\cite{R05}.

Regarding the case $m\geq 2$, Plateau experimentally established in~\cite{P1873} that soap films are made of constant mean curvature smooth surfaces meeting in threes along an edge, the so-called \emph{Plateau border}, at an angle of 120 degrees. These Plateau borders, in turn, meet in fours at a vertex at an angle of $\arccos(-\tfrac{1}{3})\simeq 109.47$ degrees (the \emph{tetrahedral angle}). Existence and regularity of minimizers of~\eqref{eq:intro_bubble_cluster_problem} in the Euclidean setting $(\R^n,P,\leb^n)$ were proved by Almgren in his celebrated work~\cite{A76}. Here, $P$ denotes the standard De Giorgi perimeter and $\mathcal L^n$ the $n$-dimensional Lebesgue measure. 
Plateau's observations were rigorously confirmed by Taylor in~\cite{T76} (a proof of the same result in higher dimensions $n\geq 4$ was announced in~\cite{W85}). The case $n=2$ was treated separately in~\cite{M94}. 

When $m=2$, \eqref{eq:intro_bubble_cluster_problem} is the \emph{double bubble problem}. This is the type of problem that we address in this paper for an anisotropic perimeter, called the Grushin perimeter, see Section \ref{ss:our}. In the Euclidean setting, the natural candidate solution is the so-called \emph{standard double bubble} given by three $(n-1)$-dimensional spherical cups intersecting in an $(n-2)$-dimensional sphere at an angle of 120 degrees (for equal volumes $v_1=v_2$, the central cup is indeed a flat disc). The first proof of this result for $n=2$ was given in~\cite{FABHZ93} exploiting the analysis carried out in~\cite{M94}. A second proof appeared in~\cite{DLSW09}. The case $n=3$ was established first in~\cite{HS00} for equal volumes and then in~\cites{HMRR02} with no restrictions. The case $n\geq 4$ was finally solved in~\cite{R08}.

The double bubble problem has been addressed also in other spaces. For $\mathbb M=\mathbb S^n$
the problem was completely solved for $n=2$ in~\cite{M96}, while for $n\geq 3$ only partial results are available, see~\cites{CF02,C-et-al08}. The double bubble problem was completely solved on the $2$-dimensional boundary of the cone in~$\R^3$, where minimizers are either two concentric circles or a circle lens, see~\cite{LB06}, and on the flat 2-torus, where five types of minimizers occur, see~\cite{C-et-al04}. 

For $m\geq 3$, problem~\eqref{eq:intro_bubble_cluster_problem} is still unsolved even in the Euclidean case and presents several interesting open questions, see~\cite{M12}*{Part~IV}. The case $\mathbb{M}=\R^2$, $m=3$ was solved in~\cite{W04}. For a detailed review on minimal partition problems, see \cites{M12,M16}.


\subsection{Our setting}\label{ss:our}
In this paper, we address problem~\eqref{eq:intro_bubble_cluster_problem} for  $m=2$, where $\mathbb M=\R^2$, $\mathsf{V}=\mathcal L^2$ is the $2$-dimensional Lebesgue measure and, for $\alpha\geq 0$, $\mathsf{P}=P_\alpha$ is the \emph{Grushin $\alpha$-perimeter} 
given by
\begin{equation}\label{eq:intro_def_Grushin_perim}
P_\alpha(E)=\sup\set*{\int_E\left(\de_x\phi_1+|x|^\alpha\de_y\phi_2\right)\,dxdy : \phi_1,\phi_2\in C^1_c(\R^2),\ \sup_{\R^2}\sqrt{\phi_1^2+\phi_2^2}\le1}
\end{equation}
for any $\mathcal L^2$-measurable set $E\subset\R^2$. For $\alpha=0$, the perimeter in~\eqref{eq:intro_def_Grushin_perim} reduces to the standard Euclidean one in~$\R^2$. The $\alpha$-perimeter is naturally associated with a Carnot--Carath\'eodory structure in~$\R^2$ called the \emph{Grushin plane}, {\em i.e.,} the manifold~$\R^2$ endowed with the vector fields $X=\de_x$ and $Y=|x|^\alpha\,\de_y$. 
An essential feature of~$P_\alpha$ 
is its invariance under (Euclidean) \emph{vertical translations} $(x,y)\mapsto(x,y+h)$ for $h\in\R$. Moreover, the Grushin plane $(\R^2,P_\alpha,\leb^2)$ is homogeneous with respect to the intrinsic anisotropic \emph{dilations} given by $(x,y)\mapsto\delta_\lambda(x,y)=(\lambda x,\lambda^{\alpha+1}y)$ for all $\lambda>0$ and $(x,y)\in\R^2$, that is, 
\begin{equation*}
\leb^2(\delta_\lambda(E))=\lambda^{\alpha+1}\leb^2(E), 
\qquad 
P_\alpha(\delta_\lambda(E))=\lambda^{\alpha+2} P_\alpha(E),
\end{equation*}
see~\cite{MM04}*{Proposition~2.2}.


The case $m=1$ in problem~\eqref{eq:intro_bubble_cluster_problem} was completely solved in this context in~\cite{MM04}, where the authors showed existence and uniqueness up to dilations and vertical translations of the \emph{Grushin isoperimetric set}. This is $E_\alpha=\set*{(x,y)\in\R^2 : |y|\le\phi_\alpha(|x|),\ |x|\le1}$, where the profile function $\phi_\alpha\colon[0,1]\to[0,+\infty)$ is given by
\begin{equation}\label{eq:isop_profile}
\phi_\alpha(x)=\int_{\arcsin x}^\frac{\pi}{2}\sin^{\alpha+1}(t)\ dt,\qquad\text{for }x\in [0,1].
\end{equation}
See also \cite{FM16} for a generalization to higher dimensional Grushin structures.

We remark that a regularity theory for almost minimizers of the $\alpha$-perimeter is not yet available. 
We refer to \cites{MV15, MS17, M15} for some partial results in the strictly related setting of Heisenberg group.
For this reason, the case $m=2$ in problem~\eqref{eq:intro_bubble_cluster_problem} cannot be addressed in full generality for the Grushin perimeter following the approach of~\cites{FABHZ93,M94}. As a matter of fact, a candidate solution of the double bubble problem in the sub-Riemannian setting has not been proposed yet. In this paper we study possible configurations and formulate a ``standard double bubble conjecture'' in this context. To this purpose, we study the case of two equal volumes and we assume that the \emph{contact interface} between the two bubbles is contained in one of the two coordinate axes.

\subsubsection{The problem for vertical interfaces}

For any $\leb^2$-measurable set $E\subset\R^2$,  let
\begin{equation}\label{eq:vertical_pm}
E^{+x}=E\cap\set*{(x,y)\in\R^2 : x>0}, 
\qquad 
E^{-x}=E\cap\set*{(x,y)\in\R^2 : x<0}.
\end{equation}
For a given $v>0$, we define the class of admissible sets as
\begin{equation}\label{eq:def_adimissible_class}
\A^x(v)=\set*{E\subset\R^2 : \text{$E$ is $\leb^2$-measurable},\ \leb^2(E^{+x})=\leb^2(E^{-x})=v}.
\end{equation} 
The first problem that we treat is
\begin{equation}\label{eq:minpart}
\inf\set*{\P_\alpha^x(E) : E\in\A^x(v)}
\end{equation}
where, for any $E\in\A^x(v)$, we let 
\begin{equation}\label{eq:Pa}
\P_\alpha^x(E)=\frac{1}{2}\big(P_\alpha(E^{+x})+P_\alpha(E^{-x})+P_\alpha(E)\big)\in[0,+\infty].
\end{equation}
When $\alpha=0$, we simply write $\P_0^x=\P^x$.

\subsubsection{The problem for horizontal interfaces}

For any $\leb^2$-measurable set $E\subset\R^2$, let
\begin{equation}\label{eq:horizontal_pm}
E^{+y}=E\cap\set*{(x,y)\in\R^2 : y>0}, 
\qquad 
E^{-y}=E\cap\set*{(x,y)\in\R^2 : y<0}.
\end{equation}
For a given $v>0$, we define the class of admissible sets as
\begin{equation}\label{eq:def_adimissible_class_y}
\A^y(v)=\set*{E\subset\R^2 : \text{$E$ is $\leb^2$-measurable},\ \leb^2(E^{+y})=\leb^2(E^{-y})=v}.
\end{equation} 
The second problem that we treat is
\begin{equation}\label{eq:minpart_y}
\inf\set*{\P_\alpha^y(E) : E\in\A^y(v)},
\end{equation}
where, for any $E\in\A^y(v)$, we let 
\begin{equation}\label{eq:Pa_y}
\P_\alpha^y(E)=\frac{1}{2}\big(P_\alpha(E^{+y})+P_\alpha(E^{-y})+P_\alpha(E)\big)\in[0,+\infty].
\end{equation}
Again, when $\alpha=0$, we simply write $\P_0^y=\P^y$.

\subsubsection{Main results}
Existence of minimizers to problems \eqref{eq:minpart} and \eqref{eq:minpart_y} is proved in Theorems \ref{th:existence} and \ref{th:existence_y} respectively.
Our approach to characterize them is based on a rearrangement technique which has its own interest (see Section \ref{sec:rearrangement} and Theorems \ref{th:reduction}, \ref{th:reduction_y}) which exploits the existence of a change of coordinates that transforms the Grushin plane $(\R^2,P_\alpha,\leb^2)$ in the \emph{transformed plane} $(\R^2,P,\M_\alpha)$. Here, $P$ is the standard Euclidean perimeter and 
$\M_\alpha$ is a weighted area, see Section \ref{ss:transf_plane}. Thanks to this rearrangement we deduce symmetry of minimizers, which yields a complete characterization of the minimal double bubbles with constrained vertical and horizontal interface by a first variation argument (see Theorems \ref{th:charact_reg_min}, \ref{th:charact_reg_min_y}). Our main results are resumed in the following theorem.

\begin{theorem}
Let $v>0$. Then solutions to problems \eqref{eq:minpart}, \eqref{eq:minpart_y} exist. Moreover, the following statements hold.

1) If $E\subset\R^2$ is a solution to \eqref{eq:minpart}, then, up to vertical translations, we have
\begin{equation*}
E=\set*{(x,y)\in\R^2 : |y|\le f(|x|),\ |x|\le r},
\end{equation*}
where $f\in C([0,r])\cap C^\infty(]0,r[)$, $r\in]0,+\infty[$,  is defined  by 
\begin{equation*}
f(x)=-\frac{1}{|k|^{\alpha+1}}\int_{-1}^{kx+\frac{1}{2}}\frac{t(\frac{1}{2}-t)^\alpha}{\sqrt{1-t^2}}\ dt,
\end{equation*} 
for all $x\in[0,r]$, where $r=-\frac{3}{2k}$ and 
\begin{equation*}
k=-\left(-\frac{2}{v}\int_{-1}^\frac{1}{2}\frac{t(\frac{1}{2}-t)^{\alpha+1}}{\sqrt{1-t^2}}dt\right)^\frac{1}{\alpha+2}.
\end{equation*}

2)
If $E\subset\R^2$ is a solution to \eqref{eq:minpart_y}, then, up to vertical translations, we have
\begin{equation*}
E=\delta_{\frac{1}{h}}\left(\set*{(x,y)\in\R^2 : \left(x,|y|-\phi_\alpha\left(\tfrac{\sqrt{3}}{2}\right)\right)\in E_\alpha}\right),
\end{equation*}
where $\phi_\alpha\colon[0,1]\to[0,+\infty[$ is the isoperimetric profile defined in~\eqref{eq:isop_profile}  and 
\begin{equation*}
h=\left[\frac{1}{v}\left(\leb^2(E_\alpha)-2\int_0^\frac{\sqrt{3}}{2}\frac{t^{\alpha+2}}{\sqrt{1-t^2}}\ dt\right)\right]^\frac{1}{\alpha+2}.
\end{equation*}
\end{theorem}

Both double bubbles with constrained interface in $(\R^2,P_\alpha,\leb^2)$ consist of three smooth curves with constant mean $\alpha$-curvature (see Figures \ref{fig:dbeucl}, \ref{fig:grushin_bubble}, \ref{fig:dbeucl_y}, \ref{fig:isop_bubble}). Even though they do not satisfy the $120$-degree Plateau's rule in the Grushin plane, they do satisfy it in the transformed plane $(\R^2,P,\M_\alpha)$, see Sections \ref{sss:angle}, \ref{sss:angle_y}. This new phenomenon gives some insights about the possible structure of the singular set of minimizers in this sub-Riemannian context.

In conclusion, by comparing the two minimal double bubbles in the case $\alpha=1$, see Remark \ref{rem:comparison}, we obtain a candidate solution to the general double bubble problem for the Grushin perimeter: the configuration with vertical interface has perimeter strictly greater than the one with horizontal interface. This establishes a connection with the standard double bubble in the Euclidean setting: in fact, the minimal double bubble with horizontal interface is obtained by translating and dilating the Grushin isoperimetric set found in~\cite{MM04}, similarly to the Euclidean case. This leads us to conjecture that this configuration may solve the double bubble problem with no assumptions on the contact interface.

%

\section{Preliminaries on the Grushin perimeter}
\label{sec:elem_prop_grushin}
\subsection{Representation formulas}
We start recalling some representation formulas for the $\alpha$-perimeter that we will use in the sequel. 
Let $E\subset\R^2$ be a bounded open set with Lipschitz boundary. Then
\begin{equation}\label{eq:repres_formula}
P_\alpha(E)=\int_{\de E}\left(N_{E,1}(x,y)^2+|x|^{2\alpha}N_{E,2}(x,y)^2\right)^{1/2}\ d\haus^1(x,y),
\end{equation}
where $N_E(x,y)=(N_{E,1}(x,y),N_{E,2}(x,y))$ is the (outward) unit normal to $\de E$ at the point $(x,y)\in\de E$. Formula~\eqref{eq:repres_formula} is proved in~\cite{MM04}*{Theorem~2.1}. 
Let $D=(a,b)\subset\R$ be a bounded open interval and let $\phi,\psi\colon D\to\R$ be two bounded Lipschitz functions on $D$. Consider the open epigraphs 
\begin{equation*}
E_\phi=\set*{(x,y)\in\R^2 : x\in D,\ y>\phi(x)},
\qquad
E_\psi=\set*{(x,y)\in\R^2 : y\in D,\ x>\psi(y)}.
\end{equation*}
The sets $E_\phi$ and $E_\psi$ have finite $\alpha$-perimeter in the cylinders $D\times\R$ and $\R\times D$ respectively. Moreover, formula~\eqref{eq:repres_formula} implies that
\begin{equation}\label{eq:perim_from_x_to_y}
P_\alpha(E_\phi;D\times\R)=\int_D\sqrt{|x|^{2\alpha}+\phi'(x)^2}\ dx,
\end{equation}
and 
\begin{equation}\label{eq:perim_from_y_to_x}
P_\alpha(E_\psi;\R\times D)=\int_D\sqrt{1+|\psi(y)|^{2\alpha}\,\psi'(y)^2}\ dy.
\end{equation}

\subsection{Transformed plane}
\label{ss:transf_plane}
As observed in~\cite{MM04}, there exists a change of coordinates that allows us to identify the Grushin plane $(\R^2,P_\alpha,\leb^2)$ with the \emph{transformed (Grushin) plane} $(\R^2,P,\M_\alpha)$, where $P$ is the Euclidean perimeter and $\M_\alpha$ is a weighted $2$-dimensional Lebesgue measure.
Precisely, consider the functions $\Phi,\Psi\colon\R^2\to\R^2$ defined as
\begin{equation}\label{eq:def_Phi_Psi}
\Phi(\xi,\eta)=\left(\sgn(\xi)|(\alpha+1)\xi|^\frac{1}{\alpha+1},\ \eta\right), 
\qquad 
\Psi(x,y)=\left(\sgn(x)\frac{|x|^{\alpha+1}}{\alpha+1},\ y\right).
\end{equation}
Clearly, the functions $\Phi$ and $\Psi$ are homeomorphisms with $\Phi^{-1}=\Psi$ and, for any $\xi\ne0$, $|\det J\Phi(\xi,\eta)|=|(\alpha+1)\xi|^{-\frac{\alpha}{\alpha+1}}$. By~\cite{MM04}*{Proposition~2.3}, for any $\leb^2$-measurable set $E\subset\R^2$, the transformed set $F=\Psi(E)$ satisfies
\begin{equation}\label{eq:prop_Phi_Psi}
P(F)=P_\alpha(E), 
\qquad 
\M_\alpha(F)=\int_F |(\alpha+1)\xi|^{-\frac{\alpha}{\alpha+1}}\ d\xi d\eta=\leb^2(E).
\end{equation}

\subsection{Grushin isoperimetric set}
\label{subsec:grushin_isop_set}

By~\cite{MM04}*{Theorem~3.2}, the following isoperimetric inequality holds. For any measurable set $E\subset\R^2$ with $\leb^2(E)<+\infty$, we have  
\begin{equation}\label{eq:isop_ineq}
\leb^2(E)\le c(\alpha)P_\alpha(E)^\frac{\alpha+2}{\alpha+1},
\end{equation}
where $c(\alpha)>0$ is a constant depending only on $\alpha\ge0$. The equality in~\eqref{eq:isop_ineq} is achieved on the \emph{Grushin isoperimetric set}
\begin{equation}\label{eq:isop_set}
E_\alpha=\set*{(x,y)\in\R^2 : |y|\le\phi_\alpha(|x|),\ |x|\le1},
\end{equation}
where the \emph{isoperimetric profile} $\phi_\alpha\colon[0,1]\to[0,r_\alpha]$ is given by \eqref{eq:isop_profile}
and we let $r_\alpha=\phi_\alpha(0)$. The isoperimetric set is unique up to dilations and vertical translations. Observe that the isoperimetric profile satisfies
\begin{equation*}
\phi_\alpha'(x)=-\frac{x^{\alpha+1}}{\sqrt{1-x^2}}
\end{equation*}
for all $x\in]0,1[$ and that
\begin{equation*}
P_\alpha(E_\alpha)=2\int_0^\pi\sin^\alpha(t)\ dt,
\qquad
\leb^2(E_\alpha)=\frac{\alpha+1}{\alpha+2}\,P_\alpha(E_\alpha).
\end{equation*}

We remark that the boundary of the isoperimetric set $E_\alpha$ is not smooth. Precisely, if $\alpha\in\N$, then  $\partial E_\alpha$ is $C^{\alpha+1}$ but not $C^{\alpha+2}$ around the $y$-axis.

\subsection{Additional terminology}
\label{ss:symmetry}
We conclude this section introducing some additional terminology.
We say that a set $E\subset\R^2$ is \emph{$x$-symmetric} (respectively, \emph{$y$-symmetric}) if $(x,y)\in E$ implies $(-x,y)\in E$ (respectively, if $(x,y)\in E$ implies $(x,-y)\in E$). For every $t\in\R$, we define the sections
\begin{equation}\label{eq:sections_notation}
E^x_t=\set*{y\in\R : (t,y)\in E}, \qquad E^y_t=\set*{x\in\R : (x,t)\in E}.
\end{equation}
The set $E$ is \emph{$x$-convex} (respectively, \emph{$y$-convex}) if the section $E^y_t$ (respectively, the section $E^x_t$) is an open interval for every $t\in\R$. Moreover, the set $E$ is \emph{$y$-Schwarz symmetric} (respectively, \emph{$x$-Schwarz symmetric}) if it is both $y$-symmetric and $y$-convex (respectively, $x$-symmetric and $x$-convex). We denote by $\S_x$ the class of $\leb^2$-measurable, $x$-symmetric sets in $\R^2$ and by $\S^*_y$ the class of $\leb^2$-measurable and $y$-Schwarz symmetric sets in $\R^2$. The classes $\S_y$ and $\S^*_x$ are analogously defined. Finally, a set $E\in\S_x$ is \emph{$x$-transformed-convex} if $\Psi(E^{+x})$ is convex. A \emph{$y$-transformed-convex} set is defined similarly.

\section{A rearrangement in the half-plane}
\label{sec:rearrangement}

\subsection{Preliminaries}

In this section we introduce a rearrangement that decreases the $\alpha$-perimeter of suitable symmetric sets in the admissible class~\eqref{eq:def_adimissible_class}. Thanks to the change of variables~\eqref{eq:def_Phi_Psi}, it is enough to work in the Euclidean plane, so that along all this section we can assume $\alpha=0$. In the following, we let $H=[0,+\infty[\times\R$ be the \emph{closed} right half-plane and $H^+=]0,+\infty[\times\R$ the \emph{open} right half-plane.

\begin{definition}\label{def:essinf_esssup}
Let $A\subset\R$ be a measurable set. The \emph{essential inf of $A$} is defined as
\begin{equation*}
\ess\inf A=\sup\set*{t\in\R : \leb^1(A\cap]-\infty,t[)=0}.
\end{equation*}
The \emph{essential sup of $A$} is defined in the analogous way. Note that $\leb^1(A)=0$ if and only if either $\ess\inf A=+\infty$ or $\ess\sup A=-\infty$. 
\end{definition}

\begin{lemma}\label{lemma:def_essinf_slice_profile}
Let $E\subset H$ be a measurable set such that $E\in\S_y^*$. Up to modify $E$ on a negligible set, the functions $\lambda_E,\phi_E\colon\R\to[0,+\infty]$ given by
\begin{equation*}
\lambda_E(t)=\leb^1(E^y_t), \qquad \phi_E(t)=\ess\inf E^y_t, \qquad t\in\R,
\end{equation*}
are even on $\R$, monotone and left-continuous on $]0,+\infty[$. In particular, $\lambda_E$ is non-increasing and $\phi_E$ is non-decreasing on $]0,+\infty[$.
\end{lemma}

\begin{proof}
Since $E$ is measurable, the horizontal section $E^y_t$ is measurable for a.e.\ $t\in\R$, so that the functions $\lambda_E,\phi_E\colon\R\to[0,+\infty]$ are well-defined a.e. Moreover, since $E\in\S_y^*$, we have $E^y_t=E^y_{-t}$ for every $t\in\R$ and $E^y_t\subset E^y_s$ for every $0<s<t$. Thus $\lambda_E$ (respectively,~$\phi_E$) is equivalent to a function even on~$\R$ and non-increasing (respectively, non-decreasing) on $]0,+\infty[$. Since a monotone function can only have countably many discontinuities in its domain, the conclusion follows. 
\end{proof}

\subsection{Rearrangement and first properties}
Let $E\subset H$ be a measurable set such that $E\in\S_y^*$. Let $\lambda_E,\phi_E\colon\R\to[0,+\infty]$ be the functions given by Lemma~\ref{lemma:def_essinf_slice_profile}. Let $\dom\phi_E$ be the set of points where~$\phi_E$ is finite. By Definition~\ref{def:essinf_esssup}, it is enough to work on $\dom\phi_E$. Let $D=\set*{d_k : k\in\N}\subset]0,+\infty[$ be the set of discontinuity points of the function $\phi_E$ in $\dom\phi_E$. For all $k\in\N$, we set 
\begin{equation}\label{eq:def_jump}
j_k=\big((\phi_E(d_k+)-\phi_E(d_k))-(\lambda_E(d_k)-\lambda_E(d_k+))\big)^+\ge0.
\end{equation}
We define the function $\tau_E\colon\R\to[0,+\infty]$ as
\begin{equation}\label{eq:def_translating_function}
\tau_E(t)=\sum_{k\in\N}j_k\chi_{(d_k,+\infty)}(|t|), \qquad t\in\R.
\end{equation}
Note that $\tau$ is even on $\R$, non-decreasing and left-continuous on $]0,+\infty[$, and such that $\tau_E(t)\le\phi_E(t)$ for $t\in\dom\phi_E$. We thus define the horizontal rearrangement of $E$ as 
\begin{equation}\label{eq:def_rearrangement}
E^\bigstar=\set*{(x,y)\in\R^2 : y\in\dom\phi_E,\ 0<x-\phi_E(y)+\tau_E(y)<\lambda_E(y)}.
\end{equation}
It is easy to see that $E^\bigstar\subset H$ is measurable and such that $E^\bigstar\in\S_y$. The following result shows that the rearrangement defined in~\eqref{eq:def_rearrangement} does not modify $x$-convex sets. 

\begin{lemma}\label{lemma:only_if_rearrang}
Let $E\subset H$ be a measurable set such that $E\in\S_y^*$. Let $E^\bigstar\subset H$ be as in~\eqref{eq:def_rearrangement}. If $E^y_t$ is equivalent to an interval for a.e.\ $t\in\R$, then $E^\bigstar=E$ up to negligible sets. 
\end{lemma}

\begin{proof}
Up to a modification of $E$ on a negligible set, we can directly assume that  
\begin{equation}\label{eq:fusco}
E=\set*{(x,y)\in\R^2 : y\in\dom\phi_E,\ 0<x-\phi_E(y)<\lambda_E(y)},
\end{equation}
see~\cite{M12}*{Lemma~14.6}. Therefore, we just need to prove that $\tau_E(y)=0$ for every $y\in\R$. To do so, let $0<s<t$ and note that $E^y_t\subset E^y_s$ because $E\in\S_y^*$. By~\eqref{eq:fusco}, this means that
\begin{equation*}
\big]\phi_E(t),\phi_E(t)+\lambda_E(t)\big[
\subset
\big]\phi_E(s),\phi_E(s)+\lambda_E(s)\big[
\end{equation*}
for $0<s<t$, so that $(\phi_E(t)-\phi_E(s))-(\lambda_E(s)-\lambda_E(t))\le0$. Recalling~\eqref{eq:def_jump} and the definition of $\tau_E$ in~\eqref{eq:def_translating_function}, this concludes the proof.
\end{proof}

\subsection{Approximation lemma and elementary inequalities}

In the proof of the main result of this section, Theorem~\ref{th:rearrang} below, we will need the following approximation result. See also~\cite{F17}*{Lemma~2.1}. 

\begin{lemma}\label{lemma:polyhedral_approx_y-Steiner}
Let $E\subset H$ be a measurable set such that $\leb^2(E)<+\infty$, $P(E)<+\infty$ and $E\in\S_y^*$. There exists a sequence $(E_k)_{k\in\N}\subset H$ of bounded open sets with polyhedral boundary such that $E_k\in\S_y^*$ and, as $k\to+\infty$,
\begin{equation*}
\chi_{E_k}\to\chi_E\ \text{in}\ L^1, 
\qquad 
P(E_k)\to P(E), 
\qquad 
P(E_k;H^+)\to P(E;H^+).
\end{equation*}
\end{lemma}

\begin{proof}
Since $\leb^2(E)<+\infty$, it is not restrictive to assume that $E$ is bounded, see~\cite{M12}*{Remarks~13.12}. Let us set
\begin{equation*}
\tilde{E}=\set*{(x,y)\in\R^2 : (|x|,y)\in E},
\end{equation*}
the symmetrization of $E$ with respect to the $y$-axis. It is immediate to see that $\tilde{E}\in\S_x\cap\S_y^*$ is bounded with $P(\tilde{E})<+\infty$ and $P(\tilde{E};\de H)=0$. By~\cite{M12}*{Theorem~13.8}, there exists a sequence $(\tilde{E}_k)_{k\in\N}$ of bounded open sets with smooth boundary such that 
\begin{equation}\label{eq:sym_approx_conv}
\chi_{\tilde{E}_k}\to\chi_{\tilde{E}}\ \text{in}\ L^1,
\qquad 
P(\tilde{E}_k)\to P(\tilde{E}),
\end{equation}
as $k\to+\infty$. Arguing as in the proof of~\cite{F17}*{Lemma~2.1}, one can prove that $(\tilde{E}_k)_{k\in\N}\subset\S_x\cap\S_y^*$. Now let us set $E_k=\tilde{E}_k\cap H^+$ for each $k\in\N$. Since $P(\tilde{E})=2P(E;H^+)$ and $P(\tilde{E}_k)=2P(E_k;H^+)$ by symmetry, from~\eqref{eq:sym_approx_conv} we deduce that
\begin{equation}\label{eq:approx_conv}
\chi_{E_k}\to\chi_E\ \text{in}\ L^1,
\qquad 
P(E_k;H^+)\to P(E;H^+),
\end{equation}
as $k\to+\infty$. We claim that $P(E_k)\to P(E)$ as $k\to+\infty$. Indeed, by~\eqref{eq:approx_conv}, it is enough to show that 
\begin{equation}\label{eq:approx_conv_traces}
P(E_k;\de H)\to P(E;\de H)
\end{equation}
as $k\to+\infty$. Recalling~\eqref{eq:approx_conv}, the limit in~\eqref{eq:approx_conv_traces} follows by the definition and the continuity of the trace operator for $BV$~functions on bounded open sets with Lipschitz boundary, see~\cite{AFP00}*{Theorems~3.87 and~3.88} (see also~\cite{EG15}*{Theorem~5.6}). By a standard approximation by linear interpolation and a diagonalization argument, one can replace $(E_k)_{k\in\N}$ with a sequence of bounded open sets with polyhedral boundaries. 
\end{proof}

In the proof of Theorem~\ref{th:rearrang} below, we will need the following elementary inequalities, which we prove here for the reader's convenience.

\begin{lemma}\label{lemma:elem_ineq}
Let $N\ge1$ and let $a_k\in\R^n$ for all $k=1,\dots,N$. Then
\begin{equation}\label{eq:elem_ineq_1}
\sum_{k=1}^N\sqrt{1+|a_k|^2}-\sqrt{1+\abs*{\sum_{k=1}^N a_k}^2}
\ge
\frac{N-1}{\sum_{k=1}^N\sqrt{1+|a_k|^2}+\sqrt{1+\abs*{\sum_{k=1}^N a_k}^2}} 
\end{equation} 
and thus, in particular,
\begin{equation}\label{eq:elem_ineq_2}
\sum_{k=1}^N\sqrt{1+|a_k|^2}-\sqrt{1+\abs*{\sum_{k=1}^N a_k}^2}
\ge
\frac{N-1}{2\sum_{k=1}^N\sqrt{1+|a_k|^2}}. 
\end{equation} 
\end{lemma}

\begin{proof}
We have
\begin{equation*}
\left(\sum_{k=1}^N\sqrt{1+|a_k|^2}\right)^2=N+\sum_{k=1}^N |a_k|^2+2\sum_{1\le h<k\le N}\sqrt{(1+|a_h|^2)(1+|a_k|^2)}
\end{equation*}
and
\begin{equation*}
1+\abs*{\sum_{k=1}^N a_k}^2=1+\sum_{k=1}^N |a_k|^2+2\sum_{1\le h<k\le N}\scalar*{a_h,a_k},
\end{equation*}
Therefore, by Cauchy-Schwarz inequality, we get
\begin{align*}
\left(\sum_{k=1}^N\sqrt{1+|a_k|^2}\right)^2&-\left(1+\abs*{\sum_{k=1}^N a_k}^2\right)=\\
&=N-1+2\sum_{1\le h<k\le N}\left(\sqrt{(1+|a_h|^2)(1+|a_k|^2)}-\scalar*{a_h,a_k}\right)\ge N-1
\end{align*}
and thus
\begin{align*}
\sum_{k=1}^N\sqrt{1+|a_k|^2}&-\sqrt{1+\abs*{\sum_{k=1}^N a_k}^2}
=
\frac{\left(\sum_{k=1}^N\sqrt{1+|a_k|^2}\right)^2-\left(1+\abs*{\sum_{k=1}^N a_k}^2\right)}{\sum_{k=1}^N\sqrt{1+|a_k|^2}+\sqrt{1+\abs*{\sum_{k=1}^N a_k}^2}}\\
&\ge
\frac{N-1}{\sum_{k=1}^N\sqrt{1+|a_k|^2}+\sqrt{1+\abs*{\sum_{k=1}^N a_k}^2}},
\end{align*} 
which is~\eqref{eq:elem_ineq_1}. In particular, we have $\sum_{k=1}^N\sqrt{1+|a_k|^2}\ge\sqrt{1+\abs*{\sum_{k=1}^N a_k}^2}$ and so from~\eqref{eq:elem_ineq_1} we deduce~\eqref{eq:elem_ineq_2}.
\end{proof}

\subsection{Rearrangement theorem}

We are now ready to prove the main result of this section. The argument follows the strategy outlined in the proof of~\cite{M12}*{Theorem~14.4}.

\begin{theorem}\label{th:rearrang}
Let $E\subset H$ be a measurable set such that $\leb^2(E)<+\infty$, $P(E)<+\infty$ and $E\in\S_y^*$. Let $E^\bigstar\subset H$ be the set defined in~\eqref{eq:def_rearrangement}. Then $\leb^2(E^\bigstar)=\leb^2(E)$, $E^\bigstar\in\S_y^*$ and 
\begin{equation}\label{eq:th_rearrang}
\min\set*{P(E)-P(E^\bigstar),\ P(E;H^+)-P(E^\bigstar;H^+)}\ge0.
\end{equation}
Moreover, equality holds in~\eqref{eq:th_rearrang} if and only if $E^y_t$ is equivalent to an interval for a.e.\ $t\in\R$, in which case $E^\bigstar=E$ up to negligible sets.
\end{theorem}

\begin{proof}
We divide the proof in two steps.

\medskip

\textit{Step~1}. Let us assume that $E$ is a bounded open set in $\R^2$ with polyhedral boundary and that the outer unit normal to $E$ (that is elementarily defined at $\haus^1$-a.e.\ point of $\de E$) is never orthogonal to $\mathrm{e}_1$. By this assumption and by the implicit function theorem, we get that 
\begin{align}
E&=\bigcup_{h=1}^M\set*{(x,y)\in\R^2 : y\in I_h^+\cup I_h^-,\ x\in\bigcup_{k=1}^{N(h)}\left]u_h^k(y), v_h^k(y)\right[},\label{eq:boundary_polyhed}\\
\de E&=\bigcup_{h=1}^M\bigcup_{k=1}^{N(h)}\Gamma(u_h^k,I^+_h\cup I^-_h)\cup\Gamma(v_h^k,I^+_h\cup I^-_h).\nonumber
\end{align}
Here $\set{I_h^\pm}_{h=1,\dots,M}$ is a finite family of non-overlapping bounded intervals such that 
\begin{equation*}
I_h^+\subset[0,+\infty[, \qquad I_h^-=-I_h^+,\qquad I=\bigcup_{h=1}^M I^+_h\cup I^-_h,
\end{equation*}
where $I=\set*{t\in\R : \leb^1(E^y_t)>0}$ and $u_h^k,v_h^k\colon\R\to[0,+\infty[$, with $h=1,\dots,M$, $k=1,\dots,N(h)$, are affine even functions on $\R$ such that $u_h^k\le v_h^k$ on $I^+_h$ and $u_h^k$ (respectively, $v_h^k$) is non-decreasing (respectively, non-increasing) on $]0,+\infty[$, see Figure~\ref{fig:rearrang_polyhed}. 
For $D\subset\R$ and $u\colon D\to\R$, we denote by $\Gamma(u,D)=\set{(z,t)\in\R^2 : z\in D, t=u(z)}$ the graph of $u$ over $D$. 

\input{rearrang_polyhed}

The function $\lambda_E\colon I\to[0,+\infty[$ of Lemma~\ref{lemma:def_essinf_slice_profile} is given by
\begin{equation*}
\lambda_E(y)=\sum_{k=1}^{N(h)}v_h^k(y)-u_h^k(y), \qquad \forall y\in I^+_h\cup I^-_h.
\end{equation*}
Note that $\lambda_E$ is a continuous function. The function $\phi_E\colon I\to[0,+\infty[$ of Lemma~\ref{lemma:def_essinf_slice_profile} is given by
\begin{equation*}
\phi_E(y)=u_h^1(y), \qquad \forall y\in I^+_h\cup I^-_h.
\end{equation*}
We claim that, since $E\in\S_y^*$, $\de E\cap\de H$ is a symmetric interval centered at the origin and
\begin{equation*}
\de E\cap\de H=\set*{y\in\R : \phi_E(y)=0}=\phi_E^{-1}(0).
\end{equation*}
Indeed, settting
\begin{equation*}
\bar{h}=\max\set*{h=1,\dots,M :\ \text{for all}\ r\le h,\ u_r^1(y)=0\ \text{for every}\ y\in I^+_r\cup I^-_r}
\end{equation*}
or $\bar{h}=0$ if the condition in brackets is empty,  we get
\begin{equation}\label{eq:traccia_E}
\de E\cap\de H=\bigcup_{h=1}^{\bar{h}} I^+_h\cup I^-_h.
\end{equation} 
Now, let us set $\inf I_h^+=d_{h-1}$, $\sup I_h^+=d_h$. By construction, the discontinuity points of~$\phi_E$ belong to the finite set $D=\set{d_h : h=1,\dots,M}$. Therefore, the function $\tau_E\colon I\to[0,+\infty[$ defined in~\eqref{eq:def_translating_function} is given by
\begin{equation*}
\tau_E(y)=\sum_{l=1}^{h-1} \big(u_{l+1}^1(d_l)-u_l^1(d_l)\big)\chi_{(d_l,+\infty)}(|y|), \qquad \forall y\in I_h^+\cup I_h^-.
\end{equation*} 
Let also 
\begin{equation*}
\tilde{u}(y)=\phi_E(y)-\tau_E(y), \qquad y\in I
\end{equation*}  
and
\begin{equation*}
\tilde{v}(y)=\tilde{u}(y)+\lambda_E(y), \qquad y\in I.
\end{equation*}
The functions $\tilde{u},\tilde{v}$ are affine, even, non-negative and continuous functions on $I$ such that $\tilde{u}\le \tilde{v}$ and $\tilde{u}$ (respectively, $\tilde{v}$) is non-decreasing (respectively, non-increasing) on $I^+=I\cap]0,+\infty[$. Recalling~\eqref{eq:def_rearrangement}, we thus find that
\begin{align*}
E^\bigstar&=\set*{(x,y)\in\R^2 : y\in I,\ x\in\left]\tilde{u}(y),\tilde{v}(y)\right[},\\
\de E^\bigstar&=\Gamma(\tilde{u},I)\cup\Gamma(\tilde{v},I).
\end{align*}
In particular, we get that $\leb^2(E^\bigstar)=\leb^2(E)$ and $E^\bigstar\in\S_y^*$. Note that
\begin{equation*}
\phi_{E^\bigstar}(y)=\tilde{u}(y)\le\phi_E(y) \qquad y\in I,
\end{equation*}
so that $\de E\cap\de H\subset\de E^\bigstar\cap\de H$. Thus
\begin{equation}\label{eq:polyhed_trace_increase}
P(E;\de H)\le P(E^\bigstar;\de H).
\end{equation}
Indeed, let 
\begin{equation*}
\tilde{h}=\max\set*{h=1,\dots,M :\ \text{for all}\ r\le h,\ u_r^1(y)\ \text{is constant on}\ I^+_r\cup I^-_r}
\end{equation*}
or $\tilde{h}=0$ if the condition in brackets is empty. Then, we get $\tilde{h}\ge\bar{h}$, and hence
\begin{equation}\label{eq:traccia_E_star}
\de E^\bigstar\cap\de H=\bigcup_{h=1}^{\tilde{h}} I^+_h\cup I^-_h=(\de E\cap\de H)\cup\bigcup_{h=\bar{h}}^{\tilde{h}} I^+_h\cup I^-_h,
\end{equation} 
see Figure~\ref{fig:trace_increase}.

\input{trace_increase}

We now prove that $P(E^\bigstar)\le P(E)$. We have
\begin{align*}
P(E)&=2\sum_{h=1}^M\int_{I_h^+}\sum_{k=1}^{N(h)}\sqrt{1+|(u_h^k)'|^2}+\sqrt{1+|(v_h^k)'|^2}\ dy,\\
P(E^\bigstar)&=2\int_{I^+}\sqrt{1+|\tilde{u}'|^2}+\sqrt{1+|\tilde{v}'|^2}\ dy\\
&=2\sum_{h=1}^M\int_{I_h^+}\left(\sqrt{1+|(u_h^1)'|^2}+\sqrt{1+\abs*{(v_h^1)'+\sum_{k=2}^{N(h)}\big((v_h^k)'-(u_h^k)'\big)}^2}\right) dy.
\end{align*}
By inequality~\eqref{eq:elem_ineq_2} in Lemma~\ref{lemma:elem_ineq}, we obtain that
\begin{align*}
\sqrt{1+|(v_h^1)'|^2}&+\sum_{k=2}^{N(h)}\sqrt{1+|(u_h^k)'|^2}+\sqrt{1+|(v_h^k)'|^2}
-
\sqrt{1+\abs*{(v_h^1)'+\sum_{k=2}^{N(h)}\big((v_h^k)'-(u_h^k)'\big)}^2}\\
&\ge\frac{N(h)-1}{\sqrt{1+|(v_h^1)'|^2}+\sum_{k=2}^{N(h)}\sqrt{1+|(u_h^k)'|^2}+\sqrt{1+|(v_h^k)'|^2}}
\end{align*}
and therefore $P(E^\bigstar)\le P(E)$ since $N(h)\ge1$ for every $h$. The inequality above allows us to deduce some more precise information. Let $D\subset I$ be the set of those $t\in I$ such that $E^y_t$ is not an interval, so that $N(h)\ge2$ if and only if $I_h^\pm\cap D=\varnothing$. Then we get
\begin{equation*}
P(E)-P(E^\bigstar)\ge 2\sum_{h=1}^M\int_{I_h^+\cap D}\frac{1}{\sqrt{1+|(v_h^1)'|^2}+\sum_{k=2}^{N(h)}\sqrt{1+|(u_h^k)'|^2}+\sqrt{1+|(v_h^k)'|^2}}\ dy.
\end{equation*}
Since clearly it holds
\begin{equation*}
P(E)\ge 2\sum_{h=1}^M\int_{I_h^+\cap D}\sqrt{1+|(v_h^1)'|^2}+\sum_{k=2}^{N(h)}\sqrt{1+|(u_h^k)'|^2}+\sqrt{1+|(v_h^k)'|^2}\ dy,
\end{equation*}
by Cauchy-Schwarz inequality we deduce that
\begin{equation}\label{eq:quantitative_ineq_polyhedra_interval}
\big(P(E)-P(E^\bigstar)\big) P(E) \ge \leb^1(D)^2.
\end{equation}
In addition, from~\eqref{eq:polyhed_trace_increase} we deduce that
\begin{equation*}
P(E;H^+)-P(E^\bigstar;H^+)\ge P(E^\bigstar;\de H)-P(E;\de H)\ge0.
\end{equation*}
We can also obtain a stronger inequality. Indeed, by~\eqref{eq:traccia_E} and~\eqref{eq:traccia_E_star}, we have
\begin{align*}
P(E;H^+)&=2\sum_{h=1}^{\bar{h}}\int_{I_h^+}\sqrt{1+|(v_h^1)'|^2}+\sum_{k=2}^{N(h)}\sqrt{1+|(u_h^k)'|^2}+\sqrt{1+|(v_h^k)'|^2}\ dy\\
&\quad+2\sum_{h=\bar{h}+1}^M\int_{I_h^+}\sum_{k=1}^{N(h)}\sqrt{1+|(u_h^k)'|^2}+\sqrt{1+|(v_h^k)'|^2}\ dy,\\
P(E^\bigstar;H^+)&=2\int_{I^+\cap\set{\tilde{u}>0}}\sqrt{1+|\tilde{u}'|^2}\ dy+2\int_{I_h^+}\sqrt{1+|\tilde{v}'|^2}\ dy\\
&=2\sum_{h=1}^{\tilde{h}}\int_{I_h^+}\sqrt{1+\abs*{(v_h^1)'+\sum_{k=2}^{N(h)}\big((v_h^k)'-(u_h^k)'\big)}^2}\ dy\\
&\quad+2\sum_{h=\tilde{h}+1}^M\int_{I_h^+}\sqrt{1+|(u_h^1)'|^2}+\sqrt{1+\abs*{(v_h^1)'+\sum_{k=2}^{N(h)}\big((v_h^k)'-(u_h^k)'\big)}^2}\ dy.
\end{align*}
Then, since $\tilde{h}\ge\bar{h}$, arguing as before we get that
\begin{equation*}
\big(P(E;H^+)-P(E^\bigstar;H^+)\big) P(E;H^+)\ge \leb^1(D)^2.
\end{equation*}

\medskip

\textit{Step~2}. Let $E\subset H$ be a measurable set such that $\leb^2(E)<+\infty$, $P(E)<+\infty$ and $E\in\S_y^*$. By Lemma~\ref{lemma:polyhedral_approx_y-Steiner}, there exists a sequence $(E_k)_{k\in\N}$ of bounded open sets with polyhedral boundary such that
\begin{equation}\label{eq:approx_polyhedral_1}
E_k\subset H, \qquad E_k\in\S_y^*, 
\end{equation}
and, as $k\to+\infty$,
\begin{equation}\label{eq:approx_polyhedral_2}
\chi_{E_k}\to\chi_E\ \text{in}\  L^1, \qquad P(E_k)\to P(E), \qquad P(E_k;H^+)\to P(E;H^+).
\end{equation}
Without loss of generality, we can assume that for every $k\in\N$ the outer unit normal to $E_k$ is never orthogonal to $\mathrm{e}_1$ while keeping~\eqref{eq:approx_polyhedral_1} and~\eqref{eq:approx_polyhedral_2}. Let us set
\begin{equation*}
I_k=\set*{y\in\R : \lambda_{E_k}(y)>0}
\end{equation*}
and $D_k\subset I_k$ the set of those $t\in I_k$ such that $(E_k)^y_t$ is not an interval. Applying \emph{step one} to each $E_k$, we get
\begin{equation}\label{eq:approx_rearrang_ineq}
P(E_k^\bigstar)\le P(E_k),\qquad P(E_k^\bigstar;H^+)\le P(E_k;H^+),
\end{equation}
\begin{equation}\label{eq:approx_rearrang_ineq_quant}
\leb^1(D_k)^2\le\min\set*{\big(P(E_k)-P(E_k^\bigstar)\big)P(E_k),\ \big(P(E_k;H^+)-P(E_k^\bigstar;H^+)\big)P(E_k;H^+)}.
\end{equation}
By Tonelli's theorem, we have
\begin{equation}\label{eq:tonelli}
\leb^2(E_k\bigtriangleup E)=\int_\R\leb^1((E_k)^y_t\bigtriangleup E^y_t)\ dt\ge\int_\R|\lambda_{E_k}(t)-\lambda_{E}(t)|\ dt=\leb^2(E_k^\bigstar\bigtriangleup E^\bigstar).
\end{equation}
In particular, $\chi_{E_k^\bigstar}\to\chi_{E^\bigstar}$ in $L^1$ and, by the lower semicontinuity of the perimeter,
\begin{equation}\label{eq:semicont_per_star_k}
P(E^\bigstar)\le\liminf_{k\to+\infty} P(E_k^\bigstar), \qquad P(E^\bigstar;H^+)\le\liminf_{k\to+\infty} P(E_k^\bigstar;H^+).
\end{equation}
By~\eqref{eq:approx_polyhedral_2} and~\eqref{eq:approx_rearrang_ineq}, from~\eqref{eq:semicont_per_star_k} we deduce~\eqref{eq:th_rearrang}. In addition, from~\eqref{eq:approx_rearrang_ineq_quant}, we get
\begin{equation*}
\limsup_{k\to+\infty}\leb^1(D_k)^2\le\min\set*{\big(P(E)-P(E^\bigstar)\big)P(E),\ \big(P(E;H^+)-P(E^\bigstar;H^+)\big)P(E;H^+)}.
\end{equation*}
In particular, if either $P(E^\bigstar)=P(E)$ or $P(E^\bigstar;H^+)=P(E;H^+)$, then $\chi_{D_k}\to0$ in $L^1(\R)$ as $k\to+\infty$. Since~\eqref{eq:tonelli} implies that $\chi_{(E_k)^y_t}\to\chi_{E^y_t}$ in $L^1(\R)$ for a.e.\ $t\in\R$, as well as $\chi_{I_k}\to\chi_I$ in $L^1(\R)$, we may apply~\cite{M12}*{Propositions~12.15 and~14.5} to find that
\begin{equation*}
\liminf_{k\to+\infty}\chi_{I_k\setminus D_k}(t)P((E_k)^y_t)\ge\chi_I(t)P(E^y_t) \qquad \text{for a.e.}\ t\in\R.
\end{equation*}
Moreover, by Fatou's lemma, we have
\begin{equation*}
\int_I P(E^y_t)\ dt\le\liminf_{k\to+\infty}\int_{I_k\setminus D_k} P((E_k)^y_t)\ dt=2\liminf_{k\to+\infty}\leb^1(I_k\setminus D_k)=2\leb^1(I).
\end{equation*}
By the one-dimensional isoperimetric inequality, $P(E^y_t)\ge2$ for a.e.\ $t\in\R$. Therefore, $P(E^y_t)=2$ for a.e.\ $t\in I$. By~\cite{M12}*{Proposition 12.13}, for every such $t$, $E^y_t$ is equivalent to an interval. Thanks to Lemma~\ref{lemma:only_if_rearrang}, this concludes the proof.
\end{proof}

\subsection{Boundary regularity}

We conclude this section showing that $x$-convex and $y$-Schwarz symmetric sets are bounded and have Lipschitz boundary. 

\begin{proposition}\label{prop:regularity_boundary_rearrang}
Let $E\subset H$ be a measurable set such that $\leb^2(E)<+\infty$, $P(E)<+\infty$ and $E\in\S_y^*$. If $E$ is $x$-convex, then $E$ is bounded and $\de E$ is piecewise Lipschitz. In particular, if $E\subset H$, then $E^\bigstar$ is bounded and $\de E^\bigstar$ is piecewise Lipschitz.
\end{proposition}

\begin{proof}
Let $\lambda_E,\phi_E\colon\R\to[0,+\infty]$ be as in Lemma~\ref{lemma:def_essinf_slice_profile}. Since $E$ is $x$-convex, up to $\leb^2$-negligible sets we can write
\begin{equation}\label{eq:write_E}
E=\set*{(x,y)\in\R^2 : y\in\dom\phi_E,\ 0\le x-\phi_E(y)\le\lambda_E(y)},
\end{equation}
see~\cite{M12}*{Lemma~14.6}. Since $E\in\S^*_y$, there exists $b\in[0,+\infty]$ such that $\dom\phi_E=]-b,b[$ up to $\leb^1$-negligible sets, so that $E\subset\R\times]-b,b[$ up to $\leb^2$-negligible sets. Thus, by~\eqref{eq:write_E} and~\cite{CCDPM14}*{Proposition~3.2}, we get $4b\le P(E)$, so that $b\le \frac{1}{4}P(E)<+\infty$. Observe that,  since $E\in\S^*_y$, we have that $0\in\dom\phi_E$. Hence, Lemma~\ref{lemma:def_essinf_slice_profile} yields $E\subset[\phi_E(0),\phi_E(0)+\lambda_E(0)]\times\R$ up to $\leb^2$-negligible sets. Moreover, up to $\leb^2$-negligible sets, we can write
\begin{equation}\label{eq:write_E_2nd}
E=\set*{(x,y)\in\R^2 : |y|\le f(x),\ x\in[\phi_E(0),\phi_E(0)+\lambda_E(0)]}
\end{equation}  
for some non-negative $f\in L^1([\phi_E(0),\phi_E(0)+\lambda_E(0)])$. Thus, again by~\cite{CCDPM14}*{Proposition~3.2}, we get $2\lambda_E(0)\le P(E)$, so that $\lambda_E(0)\le \frac{1}{2}P(E)<+\infty$. This proves that $E$ is bounded. 

Now, let us define 
\begin{align*}
F&=\set*{(x,y)\in\R^2 : y\in\dom\phi_E,\ 0\le |x|\le\phi_E(y)},\\
G&=\set*{(x,y)\in\R^2 : y\in\dom\phi_E,\ 0\le |x|\le\phi_E(y)+\lambda_E(y)}.
\end{align*}
We have $F\subset G$ and $E=(G\setminus F)\cap H$. By the properties of $\phi_E$ and $\lambda_E$, we have that $G\in\S^*_x\cap\S^*_y$ and $F\in\S_x^*\cap\S_y$. Let us set
\begin{equation*}
\tilde{F}=\set*{(x,y)\in\R^2 : (|x|,b-|y|)\in F}.
\end{equation*}  
Then $\tilde{F}\in\S_x^*\cap\S_y^*$. Arguing as in the proof of~\cite{MM04}*{Theorem~3.1} (see also~\cite{FM16}*{Section~5.1}), we conclude that $\de G$ and $\de\tilde{F}$ are both union of the image of four Lipschitz curves. As a consequence, $\de F$ is piecewise Lipschitz. This proves that $\de E$ is piecewise Lipschitz and the proof is complete.
\end{proof}

\section{The double bubble problem for vertical interfaces}
\label{sec:vertical}

%

\subsection{Existence of minimizers with vertical interface}
\label{sec:existence_vertical}

\subsubsection{Reduction to more symmetric sets}

In this section we prove the existence of solutions to problem~\eqref{eq:minpart}. The following result restricts the class of admissible sets to more symmetric ones. For the notation we refer to Section \ref{ss:symmetry}

\begin{theorem}[Reduction]\label{th:reduction}
Let $v>0$ and let $E\in\A^x(v)$ with $\P_\alpha^x(E)<+\infty$. There exists $\tilde{E}\in\A^x(v)\cap\S_x\cap\S_y^*$, bounded and $x$-transformed-convex, such that $\P_\alpha^x(\tilde{E})\le\P_\alpha^x(E)$. Moreover, in the case $\alpha>0$, $E\in\S_x$ is equivalent to a $x$-transformed-convex set if and only if $\P_\alpha^x(\tilde{E})=\P_\alpha^x(E)$. In addition, there exist $r\in]0,+\infty[$ and $f\in C([0,r])\cap\Lip_{loc}(]0,r[)$ such that
\begin{equation}\label{eq:reduction_profile}
\tilde{E}=\set*{(x,y)\in\R^2 : |y|\le f(|x|),\ |x|\le r}.
\end{equation}
\end{theorem}

\begin{proof}
We split the proof in several steps. To avoid heavy notation, during the proof we will omit the $x$-superscript for the sets $E^{\pm x}$ defined in~\eqref{eq:vertical_pm}. 

\medskip

\textit{Step~1: $x$-symmetrization}. It is not restrictive to assume that $P_\alpha(E^+)\le P_\alpha(E^-)$ (otherwise, we can reflect $E$ with respect to the $y$-axis). We thus define
\begin{equation*}
E^1=\set*{(x,y)\in\R^2: (|x|,y)\in E^+}.
\end{equation*}
Clearly, $E^1\in\A^x(v)$ is $x$-symmetric. We claim that $\P_\alpha^x(E^1)\le\P_\alpha^x(E)$. Indeed, this is equivalent to prove that $\P^\xi(F^1)\le\P^\xi(F)$, where $F=\Psi(E)$, $F^1=\Psi(E^1)$ and~$\Psi$ is as in~\eqref{eq:def_Phi_Psi}. By De Giorgi's Structure Theorem (see~\cite{M12}*{Theorem~15.9} for instance), the perimeter measure of~$F$ is given by $\mu_F=\haus^1\mres\de^*F$, where $\de^*F$ is the reduced boundary of~$F$. We thus have
\begin{align*}
\P^\xi(F)&=\frac{1}{2}\big(P(F)+P(F^+)+P(F^-)\big)\\
&=P(F^+)+P(F^-)-\min\set*{\haus^1\big(\de^*(F^+)\cap\set{\xi=0}\big),\haus^1\big(\de^*(F^-)\cap\set{\xi=0}\big)}\\
&\ge 2P(F^+)-\haus^1\big(\de^*(F^+)\cap\set{\xi=0}\big)=\P^\xi(F^1).
\end{align*}

\medskip

\textit{Step~2: vertical Steiner symmetrization}. The set $F^1=\Psi(E^1)$ is $\xi$-symmetric and satisfies $F^{1,+}=\Psi(E^{1,+})$. By~\eqref{eq:prop_Phi_Psi}, we have
\begin{equation*}
\begin{split}
&P(F^1)=P_\alpha(E^1),
\hspace*{1.5cm}
\M_\alpha(F^1)=\leb^2(E^1), \\
&P(F^{1,+})=P_\alpha(E^{1,+}),
\qquad
\M_\alpha(F^{1,+})=\leb^2(E^{1,+}).
\end{split}
\end{equation*}
Let $F^2$ be the Steiner symmetrization of $F^1$ in the $\eta$-direction, {\em i.e.},
\begin{equation*}
F^2=\set*{(t,\eta)\in\R^2 : |\eta|<\frac{\leb^1((F^1)^\xi_t)}{2}},
\end{equation*}
with the notation \eqref{eq:sections_notation}. $F^{2,+}$ is the Steiner symmetrization of $F^{1,+}$ in the $\eta$-direction. By~\cite{CCF05}*{Theorem~1.1} (see also~\cite{CCDPM14}*{Theorem~A} for a more general statement), we have $P(F^2)\le P(F^1)$ and $P(F^{2,+})\le P(F^{1^+})$, so that $\P^\xi(F^2)\le\P^\xi(F^1)$. Moreover, if $\P^\xi(F^2)<\P^\xi(F^1)$, then $F^1$ (or, equivalently, $F^{1,+}$) is not equivalent to an $\eta$-convex set. In addition, since $\leb^1((F^1)^\xi_t)=\leb^1((F^2)^\xi_t)$ for all $t\in\R$, by Tonelli's theorem we get
\begin{align*}
\M_\alpha(F^{1,+})&=\int_{F^{1,+}} |(\alpha+1)t|^{-\frac{\alpha}{\alpha+1}}\ dt d\eta
 =\int_0^{+\infty} |(\alpha+1)t|^{-\frac{\alpha}{\alpha+1}}\,\leb^1((F^{1,+})^\xi_t)\ dt\\
&=\int_{-\infty}^{+\infty} |(\alpha+1)\xi|^{-\frac{\alpha}{\alpha+1}}\,\leb^1((F^{2,+})^\xi_t)\ dt
=\M_\alpha(F^{2,+}).
\end{align*} 
Letting $E^2=\Phi(F^2)$, where~$\Phi$ is as in~\eqref{eq:def_Phi_Psi}, we get that $E^2\in\A^x(v)\cap\S_x\cap\S^*_y$ satisfies $\P_\alpha^x(E^2)\le\P_\alpha^x(E^1)$, with strict inequality if $E^1$ is not equivalent to a $y$-convex set.

\medskip

\textit{Step~3: horizontal rearrangement}. Let $F^2=\Psi(E^2)$. Note that $F^{2,+}\subset\set*{\xi\ge0}$ is such that $\M_\alpha(F^{2,+})<+\infty$, $P(F^{2,+})<+\infty$ and $F^{2,+}\in\S_\eta^*$. We claim that $\leb^2(F_2)<+\infty$. Indeed, let $G$ be the Steiner symmetrization of $F^2$ in the $\xi$-direction, {\em i.e.,}
\begin{equation*}
G^1=\set*{(\xi,t)\in\R^2 : |\xi|<\frac{\leb^1((F^{2,+})^\eta_t)}{2}},
\end{equation*}
with the notation~\eqref{eq:sections_notation}.
Arguing as in the proof of~\cite{MM04}*{Theorem~3.1}, we have $\M_\alpha(G^1)\ge\M_\alpha(F^{2,+})$. By~\cite{CCF05}*{Theorem~1.1}, we have $P(G^1)\le P(F^{2,+})$. In addition, by Tonelli's theorem, we also have $\leb^2(G^1)=\leb^2(F^{2,+})$. Now, if $P(G^1)=P(F^{2,+})$, then either $\leb^2(F^{2,+})<+\infty$ (and we have nothing to prove) or $F^{2,+}$ is equivalent to~$\R^2$. But if $F^{2,+}$ is equivalent to~$\R^2$, then $\M_\alpha(F^{2,+})=+\infty$, which is a contradiction. So, assume that $P(G^1)<P(F^{2,+})$. The set~$G^1$ is equivalent to an open $\xi$-Schwarz and $\eta$-Schwarz symmetric set. Again, arguing as in the proof of~\cite{MM04}*{Theorem~3.1}, its convex hull $G^2=\co(G^1)$ is open and satisfies $\M_\alpha(G^2)\ge\M_\alpha(G^1)$, $\leb^2(G^2)\ge\leb^2(G^1)$ and $P(G^2)\le P(G^1)$. Since any open convex set with finite perimeter is bounded, the set $G^2$ is bounded and thus $\leb^2(G^2)<+\infty$, which immediately implies that $\leb^2(F^{2,+})<+\infty$ as claimed.

We can thus apply Theorem~\ref{th:rearrang} to the set $F^{2,+}$. We define
\begin{equation*}
F^3=\set*{(\xi,\eta)\in\R^2 : (\xi,|\eta|)\in F^{3,+}}
\end{equation*}
where $F^{3,+}=(F^{2,+})^\bigstar$. Then $F^{3,+}\in\S_\eta^*$ is $\xi$-convex, $\leb^2(F^{3,+})=\leb^2(F^{2,+})$ and
\begin{equation*}
\min\set*{P(F^{2,+})-P(F^{3,+}),P(F^{2,+};\set{\xi>0})-P(F^{3,+};\set{\xi>0})}\ge0,
\end{equation*}
with equality if and only if $F^{2,+}$ is equivalent to a $\xi$-convex set. Therefore $\P^\xi(F^3)\le\P^\xi(F^2)$
with equality if and only if $F^{2,+}$ is equivalent to a $\xi$-convex set. We claim that $\M_\alpha(F^{3,+})\ge\M_\alpha(F^{2,+})$. Indeed, by the definition of the horizontal rearrangement in~\eqref{eq:def_rearrangement}, we have that
\begin{equation}\label{eq:horizontal_rearrang_sections}
0\le\ess\inf(F^{3,+})^\eta_t\le\ess\inf(F^{2,+})^\eta_t, 
\qquad
\leb^2((F^{3,+})^\eta_t)=\leb^2((F^{2,+})^\eta_t),
\end{equation}
for a.e.\ $t\in\R$. Now let $A\subset[0,+\infty[$ be a $\leb^1$-measurable set with finite measure and let $B=]a,a+\leb^1(A)[$ for some $0\le a\le\ess\inf A$. Since the function $\xi\mapsto\xi^\beta$, with $\beta=-\frac{\alpha}{\alpha+1}$ and $\xi\ge0$, is decreasing, we have
\begin{equation}\label{eq:falcone}
\int_A\xi^\beta\ d\xi\le\int_B\xi^\beta\ d\xi.
\end{equation}
Indeed, note that inequality~\eqref{eq:falcone} is immediate if $A$ is an interval, since in this case $B=A-a$. Thus we can directly assume that $a=\ess\inf A$. Note that inequality~\eqref{eq:falcone} follows if we prove that $\int_{A\setminus B}\xi^\beta\ d\xi\le\int_{B\setminus A}\xi^\beta\ d\xi$, since
\begin{equation*}
\int_A\xi^\beta\ d\xi=\int_{A\cap B}\xi^\beta\ d\xi+\int_{A\setminus B}\xi^\beta\ d\xi\le\int_{A\cap B}\xi^\beta\ d\xi+\int_{B\setminus A}\xi^\beta\ d\xi=\int_B\xi^\beta\ d\xi.
\end{equation*}
On the other hand, since $\leb^1(A)=\leb^1(B)$ and $\ess\inf (A\setminus B)\ge\ess\sup (B\setminus A)$, we have that
\begin{align*}
\int_{A\setminus B}\xi^\beta\ d\xi\le\leb^1(A\setminus B)\,(\ess\inf (A\setminus B))^\beta\le\leb^1(B\setminus A)\,(\ess\sup (B\setminus A))^\beta\le\int_{B\setminus A}\xi^\beta\ d\xi.
\end{align*}   
Thus, by Tonelli's theorem and~\eqref{eq:horizontal_rearrang_sections}, we conclude that $\M_\alpha(F^{3,+})\ge\M_\alpha(F^{2,+})$ as claimed. In conclusion, setting $E^3=\Phi(F^3)$, we have that $E^3\in\A^x(v)\cap\S_x\cap\S_y^*$ is such that $E^{3,+}$ is $x$-convex and satisfies $\leb^2(E^3)\ge\leb^2(E^2)$ and $\P_\alpha^x(E^3)\le\P_\alpha^x(E^2)$, with strict inequality if and only if $E^{2,+}$ is not equivalent to a $x$-convex set. 

\medskip

\textit{Step~4: convexification and regularization}. Let $F^3=\Psi(E^3)$ as in the previous step. Since $F^{3,+}\in\S_y^*$ is $\xi$-convex, has finite $\leb^2$-measure and finite perimeter, by Proposition~\ref{prop:regularity_boundary_rearrang}, $F^{3,+}$ is bounded and $\de F^{3,+}$ is piecewise Lipschitz. Arguing as in the proof of~\cite{MM04}*{Theorem~3.1}, the convex hull $F^{4,+}=\co(F^{3,+})$ is $\xi$-convex, $y$-Schwarz symmetric and satisfies
\begin{equation}\label{eq:convexification}
P(F^{4,+})\le P(F^{3,+}), 
\qquad 
P(F^{4,+};\set*{\xi>0})\le P(F^{3,+};\set*{\xi>0})
\end{equation}
and, since $F^{3,+}\subset F^{4,+}$,
\begin{equation*}
\M_\alpha(F^{4,+})\ge\M_\alpha(F^{3,+}).
\end{equation*} 
Note that the inequalities in~\eqref{eq:convexification} are strict if $F^{3,+}$ is not equivalent to a convex set. Let us set
\begin{equation*}
F^4=\set*{(\xi,\eta)\in\R^2 : (|\xi|,\eta)\in F^{4,+}}
\end{equation*}
and $E^4=\Phi(F^4)$. Then $E^4\in\S_x\cap\S_y^*$ is $x$-transformed-convex and satisfies 
\begin{equation*}
\P_\alpha^x(E^4)\le\P_\alpha^x(E^3),
\qquad
\leb^2(E^4)\ge\leb^2(E^3).
\end{equation*}
Moreover, since $F^{4,+}$ is convex, there exist $0\le a<b<+\infty$ and $\phi\in C([a,b])\cap\Lip_{loc}(]a,b[)$ such that
\begin{equation*}
F^4=\set*{(\xi,\eta)\in\R^2 : |\eta|\le \phi(|\xi|),\ a\le|\xi|\le b}.
\end{equation*}
We claim that the set
\begin{gather*}
F^5=\set*{(\xi,\eta)\in\R^2 : |\eta|\le \phi(|\xi|+a),\ |\xi|\le b-a}
\end{gather*}
satisfies $\P^\xi(F^5)\le\P^\xi(F^4)$, with strict inequality if ${\min\{a,\varphi(a)\}}>0$. To prove the claim, recall that $\P^\xi(F^i)=\frac{1}{2}(P(F^{i,+})+P(F^{i,-})+P(F^i)))$ for $i=4,5$ and notice that $F^5=F^{5,+}\cup F^{5,-}$, where $F^{5,+}=F^{4,+}-(a,0)$, $F^{5,-}=F^{4,-}+(a,0)$. Since $P$ is invariant under translations, we have that
\begin{equation}
\label{eq:trans1}
P(F^{5,+})=P(F^{4,+}),\qquad P(F^{5,-})=P(F^{4,-}).
\end{equation}
Moreover, we have 
\begin{equation*}
P(F^5)=2P(F^{5,+};\{\xi>0\})=2P(F^{4,+};\{\xi>a\}).
\end{equation*}
If $a=0$, the latter is equal to $P(F^4)$. On the other hand, if $a>0$, there holds
\begin{equation}
\label{eq:trans2}
\begin{split}
P(F^4)&=2\left[P(F^{4,+};\{\xi>a\})+2\mathcal H^1(F^{4,+}\cap\{\xi=a\})\right]\\
&=2\left[P(F^{4,+};\{\xi>a\})+2\varphi(a)\right]\\
&\ge 2P(F^{4,+};\{\xi>a\})=P(F^5),
\end{split}
\end{equation}
where the last inequality is strict if also $\varphi(a)>0$. By \eqref{eq:trans1} and  \eqref{eq:trans2} the claim follows.
By Tonelli's theorem, it is easy to see that $\M_\alpha(F^5)\ge\M_\alpha(F^4)$, with strict inequality if $a>0$. We set $E^5=\Psi(F^5)$. Note that $E^5\in\S_x\cap\S_y^*$ is such that $E^{5,+}$ is $x$-convex, $x$-transformed-convex and of the form~\eqref{eq:reduction_profile}. Moreover it satisfies $\leb^2(E^5)\ge\leb^2(E^4)$, $\P_\alpha^x(E^5)\le\P_\alpha^x(E^4)$, with strict inequalities if $E^4$ was not already of the form~\eqref{eq:reduction_profile}, see also \cite{F17}*{Lemma 2.2, Step 1} for an alternative proof.   

\medskip

\textit{Step~5: scaling and conclusion}. We can now set $\tilde{E}=\delta_\lambda(E^5)$ with $\lambda>0$ such that $\leb^2(\tilde{E})=\leb^2(E)$. In fact, we have
\begin{equation*}
\lambda^{-d}\leb^2(\tilde{E})=\leb^2(E^5)=\leb^2(E^4)\ge\leb^2(E^3)\ge\leb^2(E^2)=\leb^2(E^1)=\leb^2(E) 
\end{equation*}
and
\begin{equation*}
\lambda^{1-d}\P_\alpha^x(\tilde{E})\le\P_\alpha^x(E^5)\le\P_\alpha^x(E^4)\le\P_\alpha^x(E^3)\le\P_\alpha^x(E^2)\le\P_\alpha^x(E^1)\le\P_\alpha^x(E).
\end{equation*}
Therefore, we must have $\lambda\in(0,1]$, with $\lambda<1$ if $E\in\S_x$ is not equivalent to a $x$-transformed-convex set in the case $\alpha>0$. This concludes the proof.
\end{proof}

\subsubsection{Existence of minimal double bubbles with vertical interface}

We are now ready to prove the existence of minimizers to problem~\eqref{eq:minpart}.

\begin{theorem}[Existence of minimizers with vertical interface]\label{th:existence}
Let $\alpha\ge0$ and fix $v>0$. There exists a solution $E^*\in\A^x(v)\cap\S_x\cap\S_y^*$, bounded and $x$-transformed-convex, to the minimal partition problem~\eqref{eq:minpart}. Moreover, there exist $r\in]0,+\infty[$ and $f\in C([0,r])\cap\Lip_{loc}(]0,r[)$ such that
\begin{equation*}
E^*=\set*{(x,y)\in\R^2 : |y|\le f(|x|),\ |x|\le r}.
\end{equation*}
\end{theorem} 

\begin{proof}
We study the existence of solutions by the direct method of the calculus of variation. By Theorem~\ref{th:reduction}, the class of admissible sets can be restricted to
\begin{equation*}
\mathcal{B}^x(v)=\set*{E\in\A^x(v) : E\in\S_x\cap\S_y^*\ \text{bounded, $x$-transformed-convex, as in~\eqref{eq:reduction_profile}}}.
\end{equation*}
By the isoperimetric inequality~\eqref{eq:isop_ineq}, for any $E\in\mathcal{B}^x(v)$ we have that $\P_\alpha^x(E)\le k$ for some constant $k=k(\alpha,v)\in]0,+\infty[$ depending only on~$\alpha\ge0$ and~$v>0$. 

We claim that any $E\in\mathcal{B}^x(v)$ is contained in the rectangle $[-a,a]\times[-b,b]$ for some $a,b\in]0,+\infty[$ depending only on $\alpha,v>0$. Fix $E\in\mathcal{B}^x(v)$. By Theorem~\ref{th:reduction}, there exist $r_E\in]0,+\infty[$ and $f_E\in C([0,r_E])\cap\Lip_{loc}(]0,r_E[)$ such that 
\begin{equation*}
E=\set*{(x,y)\in\R^2 : |y|\le f_E(|x|),\ |x|\le r_E}.
\end{equation*}
Thus, by the representation formula~\eqref{eq:perim_from_x_to_y}, we get
\begin{equation*}
k\ge P_\alpha(E^{+x})
\ge 2\int_0^{r_E}\sqrt{x^{2\alpha}+f_E'(x)^2}\ dx
\ge 2\int_0^{r_E} x^{\alpha}\ dx
=2\,\frac{r_E^{\alpha+1}}{\alpha+1}
\end{equation*}
and thus $2a^{\alpha+1}\le(\alpha+1)k$. In particular, the convex set $F^{+\xi}=\Phi(E^{+x})$ is bounded and contained in $[0,\bar{a}]\times\R$ for some $\bar{a}$ depending on~$a$ and~$\alpha$. Let
\begin{equation*}
\bar{b}:=\max\set*{\leb^1(F^\xi_t) : t\in[0,\bar{a}]}.
\end{equation*}
Then, by convexity, we get $k\ge P(F^{+\xi})\ge\sqrt{\bar{a}^2+2\bar{b}^2}$, which immediately implies that~$b$ depends only on~$\alpha,v>0$.

Now, let $(E_h)_{h\in\N}\subset\mathcal{B}^x(v)$ be a minimizing sequence for the problem~\eqref{eq:minpart}, namely
\begin{equation*}
\P_\alpha^x(E_h)\le C_{MP}^x\left(1+\tfrac{1}{h}\right),
\end{equation*}
where $C_{MP}^x:=\inf\set*{\P_\alpha^x(E): E\in\A^x(v)}$. Note that $C_{MP}^x>0$ because of the isoperimetric inequality~\eqref{eq:isop_ineq}. The sets $F_h=\Psi(E_h)$ are contained in the bounded set $\Phi([-a,a]\times[-b,b])$. Moreover, by~\eqref{eq:prop_Phi_Psi}, we have $\P^x(F_h)=\P_\alpha^x(E_h)\le k$ for all $h\in\N$. The space of function of bounded variation $BV(\R^2)$ is compactly embedded in $L^1_{loc}(\R^2)$ and therefore, possibly extracting a subsequence, there exists a measurable set $F\subset\Phi([-a,a]\times[-b,b])$ such that $\chi_{F_h}\to\chi_F$ a.e.\ and in $L^1(\R^2)$. Letting $E=\Phi(F)$, it follows that $\chi_{E_h}\to\chi_E$ a.e.\ and in $L^1(\R^2)$. Up to negligible sets, we have that $E\in\mathcal{B}^x(v)$. Thus there exist $r\in]0,+\infty[$ and $f\in C([0,r])\cap\Lip_{loc}(]0,r[)$ such that
\begin{equation*}
E=\set*{(x,y)\in\R^2 : |y|\le f(|x|),\ |x|\le r}.
\end{equation*}
Moreover, by the lower semicontinuity of the $\alpha$-perimeter, we have
\begin{equation*}
\P_\alpha^x(E)\le\liminf_{h\to+\infty} \P_\alpha^x(E_h)=C_{MP}^x,
\end{equation*}
which implies that $E$ is a minimum. This concludes the proof.   
\end{proof}
\subsection{Characterization of minimizers with vertical interface}
\label{sec:characterization_vertical}

\subsubsection{Regular minimizers}
\label{sss:regular}
In this section, we solve the minimal partition problem~\eqref{eq:minpart}. In Theorem~\ref{th:existence}, we proved the existence of a minimizer $E\in\A^x(v)\cap\S_x\cap\S_y^*$ that is bounded, $x$-transformed-convex and of the form
\begin{equation}\label{eq:regular_min}
E=\set*{(x,y)\in\R^2 : |y|\le f(|x|),\ |x|\le r}
\end{equation}
for some \emph{$x$-profile function} $f\in C([0,r])\cap\Lip_{loc}(]0,r[)$ with $r\in]0,+\infty[$. We call such a  minimizer a \emph{$x$-regular minimizer}. If $E$ is a $x$-regular minimizer, then 
\begin{equation}\label{eq:mP_v}
\P_\alpha^x(E)=2f(0)+4f(r)+4\int_0^r\sqrt{x^{2\alpha}+f'(x)^2}\ dx.
\end{equation}
where $\P_\alpha^x$ is defined in~\eqref{eq:Pa}.
By Theorem~\ref{th:reduction}, any minimizer $E\in\A^x(v)\cap\S_x$ of problem~\eqref{eq:minpart} is a $x$-regular minimizer.  Lemma~\ref{lemma:min_are_x-simm} below guarantees that if there exists a unique $x$-symmetric minimizer up to vertical translations, then this must be the unique minimizer of problem~\eqref{eq:minpart} up to vertical translations.

\begin{lemma}\label{lemma:min_are_x-simm}
Assume that problem~\eqref{eq:minpart} has a unique $x$-symmetric minimizer $E_0\in\A(v)\cap\S_x$ up to vertical translations. Then $E_0$ is the unique minimizer of problem~\eqref{eq:minpart} up to vertical translations. 
\end{lemma}

\begin{proof}
Let $E\in\A(v)$ be a minimizer of problem~\eqref{eq:minpart}. Consider the set
\begin{equation*}
E^*=\set*{(x,y)\in\R^2 : (|x|,y)\in E^{+x}}.
\end{equation*}
Then $E^*\in\A^x(v)\cap\S_x$ satisfies $\P_\alpha^x(E^*)\le\P_\alpha^x(E)$. By the minimality of $E$, we must have $\P_\alpha^x(E^*)=\P_\alpha^x(E)$. Thus $E^*$ is also a minimizer of problem~\eqref{eq:minpart} and it is $x$-symmetric, so $E^*=E_0$ up to vertical translations. In particular, $E^{+x}=E_0^{+x}$ up to vertical translations. With a similar argument, we also get $E^{-x}=E_0^{-x}$ up to vertical translations. Since $E_0$ is $x$-symmetric, this implies that $E=E_0$ up to vertical translations.
\end{proof}

\subsubsection{Characterization and examples}

We are now ready for the main result of this section. In Theorem~\ref{th:charact_reg_min} below, we prove that, given $\alpha\ge0$ and $v>0$, the $x$-regular minimizer of problem~\eqref{eq:minpart} is unique and has smooth boundary far from the $y$-axis. By Lemma~\ref{lemma:min_are_x-simm} and Section \ref{sss:regular}, this is the unique minimizer of problem~\eqref{eq:minpart} up to vertical translations. 

\begin{theorem}[Characterization of minimal double bubbles with vertical interface]\label{th:charact_reg_min}
Let $\alpha\ge0$, $v>0$ and let $E\in\mathcal A^x(v)\cap\S_x\cap\S_y^*$ be a regular minimizer of problem~\eqref{eq:minpart} as in~\eqref{eq:regular_min} for some profile function $f\in C([0,r])\cap\Lip_{loc}(]0,r[)$ with $r\in]0,+\infty[$. Then $E$ is unique and its profile function is given by
\begin{equation}\label{eq:th_expr_f_k}
f(x)=-\frac{1}{|k|^{\alpha+1}}\int_{-1}^{kx+\frac{1}{2}}\frac{t(\frac{1}{2}-t)^\alpha}{\sqrt{1-t^2}}\ dt
\end{equation} 
for all $x\in[0,r]$, where
\begin{equation}\label{eq:th_expr_k}
k=-\left(-\frac{2}{v}\int_{-1}^\frac{1}{2}\frac{t(\frac{1}{2}-t)^{\alpha+1}}{\sqrt{1-t^2}}dt\right)^\frac{1}{\alpha+2}.
\end{equation}
In particular, we have $f\in C^\infty(]0,r[)$ and $r=-\frac{3}{2k}$. Moreover, $f$ satisfies $f(0)>0$, $f(r)=0$, $f'(r)=-\infty$. In addition, if $\alpha>0$ then $f'(0)=0$ and $f$ has a strict maximum at $x=-\frac{1}{2k}$. Finally, the minimum of problem~\eqref{eq:minpart} is given by
\begin{equation}\label{eq:th_expr_P_alpha}
\P_\alpha^x(E)=\frac{2}{|k|^{\alpha+1}}\int_{-1}^{\frac{1}{2}}\frac{(2-t)(\frac{1}{2}-t)^\alpha}{\sqrt{1-t^2}}\ dt.
\end{equation}
\end{theorem}

\begin{proof}
We split the proof in several steps.

\medskip

\textit{Step~1: differential equation for the profile}.
We perform a first variation argument. Let $\psi\in C_c^\infty(]0,r[)$ be such that $\int_0^r\psi\ dx=0$. For $\eps\in\R$ small, consider the set 
\begin{equation*}
E_\eps=\set*{(x,y)\in\R^2 : |y|\le f(|x|)+\eps\psi(|x|),\ |x|\le r}.
\end{equation*}
Note that $E_\eps\in\A^x(v)$ for all $\eps\in\R$ small, since $f(x)>0$ for every $x\in]0,r[$ by definition.
Then, by~\eqref{eq:mP_v} and the minimality of $E$, after an integration by parts we find
\begin{equation*}
0=\frac{d\P_\alpha^x(E_\eps)}{d\eps}\bigg|_{\eps=0}
=4\int_0^r\frac{d}{dx}\left(\frac{f'}{\sqrt{(f')^2+x^{2\alpha}}}\right)\psi\ dx.
\end{equation*}
Thus, by the fundamental lemma of the Calculus of Variations, there exists a constant $k\in\R$ such that
\begin{equation}\label{eq:def_k}
\frac{d}{dx}\left(\frac{f'}{\sqrt{(f')^2+x^{2\alpha}}}\right)=k
\end{equation}
for all $x\in]0,r[$. Integrating, we get 
\begin{equation}\label{eq:def_d}
\frac{f'}{\sqrt{(f')^2+x^{2\alpha}}}=kx+d
\end{equation}
for all $x\in]0,r[$. Thus $f'$ exists at every point $x\in]0,r[$ and satisfies $\sgn f'(x)=\sgn(kx+d)$ for all $x\in]0,r[$. Therefore
\begin{equation*}
f'(x)^2\,\big(1-(kx+d)^2\big)=x^{2\alpha}(kx+d)^2
\end{equation*}
for all $x\in]0,r[$, which implies $|kx+d|<1$ for all $x\in]0,r[$. In particular, we deduce that $|d|\le 1$ and 
\begin{equation}\label{eq:f'_kd}
f'(x)=\frac{x^\alpha(kx+d)}{\sqrt{1-(kx+d)^2}}
\end{equation}
for all $x\in]0,r[$. As a consequence, $f\in C^\infty(]0,r[)$. 
By the regularity theory of $\Lambda$-minimizer of perimeter, the boundary $\de E$ is smooth far from the $y$-axis. Therefore, we must have
\begin{equation*}
f(r)=0
\qquad\text{and}\qquad
f'(r)=\lim_{x\to r} f'(x)=-\infty.
\end{equation*} 
By the expression of~$f'$ in~\eqref{eq:f'_kd}, this implies that $k\ne0$ and $kr+d=-1$. As a consequence, since $r>0$ and $|d|\le1$, we get $k<0$ and $d\in]-1,1]$.

\medskip

\textit{Step~2: proof of $f(0)>0$}. Passing to the limit in~\eqref{eq:f'_kd} as $x\to0$ we must have that 
\begin{equation*}
f'(0):=\lim_{x\to0} f'(x)\in[0,+\infty].
\end{equation*}
Assume that $f(0)=0$ by contradiction. We distinguish two cases.

\smallskip

\textit{Case~A: $f'(0)\in]0,+\infty]$}. For any $\eps>0$ sufficiently small, we define $f_\eps\colon[0,r]\to[0,+\infty[$ by setting 
\begin{equation*}
f_\eps(x)=
\begin{cases}
f(\eps) & x\in[0,\eps[\\
f(x) & x\in[\eps,r].
\end{cases}
\end{equation*}
We then consider the set 
\begin{equation*}
E_\eps:=\set*{(x,y)\in\R^2 : |y|\le f_\eps(|x|),\ |x|\le r}.
\end{equation*}
Note that
\begin{equation*}
\P_\alpha^x(E_\eps)=\P_\alpha^x(E)+2f(\eps)+4\int_0^\eps x^\alpha-\sqrt{x^{2\alpha}+f'(x)^2}\ dx
\end{equation*}
and
\begin{equation*}
\leb^2(E_\eps)=\leb^2(E)+4\int_0^\eps f(\eps)-f(x)\ dx.
\end{equation*}
We thus define
\begin{equation*}
\lambda_\eps=\left(\frac{\leb^2(E)}{\leb^2(E_\eps)}\right)^\frac{1}{\alpha+2}<1
\qquad\text{and}\qquad
F_\eps=\delta_{\lambda_\eps}(E_\eps).
\end{equation*}
Then $F_\eps\in\A^x(v)$ and $\P_\alpha^x(F_\eps)=\lambda_\eps^{\alpha+1}\P_\alpha^x(E_\eps)$. We claim that $\P_\alpha^x(F_\eps)<\P_\alpha^x(E)$ for any $\eps>0$ sufficiently small. A direct computation gives 
\begin{equation*}
\frac{d\P_\alpha^x(F_\eps)}{d\eps}\bigg|_{\eps=0}=-2f'(0)\in[-\infty,0[
\end{equation*}
and the claim follows from the Taylor's expansion of the function $\eps\mapsto\P_\alpha^x(F_\eps)$. But this contradicts the minimality of~$E$. 

\smallskip

\textit{Case~B: $f'(0)=0$}. For any $\eps>0$ sufficiently small, we define $f_\eps\colon[0,r]\to[0,+\infty[$ by setting 
\begin{equation*}
f_\eps(x)=
\begin{cases}
f(x+\eps) & x\in[0,r-\eps[\\
0 & x\in[r-\eps,r].
\end{cases}
\end{equation*}
We then consider the set 
\begin{equation*}
E_\eps:=\set*{(x,y)\in\R^2 : |y|\le f_\eps(|x|),\ |x|\le r}.
\end{equation*}
Note that
\begin{equation*}
\P_\alpha^x(E_\eps)=\P_\alpha^x(E)+2f(\eps)+4\int_\eps^r\sqrt{(x-\eps)^{2\alpha}+f'(x)^2}\, dx-4\int_0^r\sqrt{x^{2\alpha}+f'(x)^2}\, dx
\end{equation*}
and
\begin{equation*}
\leb^2(E_\eps)=\leb^2(E)-4\int_0^\eps f(x)\ dx.
\end{equation*}
We thus define
\begin{equation*}
\lambda_\eps=\left(\frac{\leb^2(E)}{\leb^2(E_\eps)}\right)^\frac{1}{\alpha+2}>1
\qquad\text{and}\qquad
F_\eps=\delta_{\lambda_\eps}(E_\eps).
\end{equation*}
Then $F_\eps\in\A^x(v)$ and $\P_\alpha^x(F_\eps)=\lambda_\eps^{\alpha+1}\P_\alpha^x(E_\eps)$. We claim that $\P_\alpha^x(F_\eps)<\P_\alpha^x(E)$ for any $\eps>0$ sufficiently small. A direct computation gives 
\begin{equation*}
\frac{d\P_\alpha^x(F_\eps)}{d\eps}\bigg|_{\eps=0}=-4\int_0^r\frac{\alpha\,x^{2\alpha-1}}{\sqrt{x^{2\alpha}+f'(x)^2}}\ dx<0
\end{equation*}
and the claim follows from the Taylor's expansion of the function $\eps\mapsto\P_\alpha(F_\eps)$. But this contradicts the minimality of~$E$. 

\medskip

\textit{Step~3: proof of $d=\frac{1}{2}$ and $f'(0)=0$}. 
We perform a first variation argument. Let $\psi\in C^\infty(]0,r[)$ be such that $\int_0^r\psi\ dx=0$ and $\psi(r)=0$. For $\eps\in\R$ small, consider the set 
\begin{equation*}
E_\eps=\set*{(x,y)\in\R^2 : |y|\le f(|x|)+\eps\psi(|x|),\ |x|\le r}.
\end{equation*}
Note that $E_\eps\in\A^x(v)$ for all $\eps\in\R$ small, since $f(x)>0$ for every $x\in[0,r[$ by \textit{step~2}.
Then, by~\eqref{eq:mP_v} and the minimality of $E$, similarly as in \textit{step~1}, we find
\begin{equation*}
0=\frac{d\P_\alpha(E_\eps)}{d\eps}\bigg|_{\eps=0}
=2\psi(0)+4\int_0^r\frac{f'\,\psi'}{\sqrt{(f')^2+x^{2\alpha}}}\ dx
=(2-4d)\,\psi(0).
\end{equation*}
The last equality follows by an integration by parts recalling~\eqref{eq:def_k}, \eqref{eq:def_d} and the assumptions on~$\psi$. By the arbitrariness of~$\psi$, we find $d=\frac{1}{2}$. Recalling the expression of $f'$ in~\eqref{eq:expr_f_k}, we get $f'(0)=\lim_{x\to0}f'(x)=0$ for all $\alpha>0$.

\medskip

\textit{Step~4: characterization of the profile}. 
By the expression of $f'$ in~\eqref{eq:f'_kd}, for $x\in[0,r]$ we can compute 
\begin{equation}\label{eq:expr_f_k}
\begin{split}
f(x)&=f(x)-f(r)
=-\int_x^r f'(t)\ dt=\\
&=-\int_x^r\frac{t^{\alpha}(kt+\frac{1}{2})}{\sqrt{1-(kt+\frac{1}{2})^2}}\ dt
=-\frac{1}{|k|^{\alpha+1}}\int_{-1}^{kx+\frac{1}{2}}\frac{t(\frac{1}{2}-t)^\alpha}{\sqrt{1-t^2}}\ dt
\end{split}
\end{equation}
using the information $kr+\frac{1}{2}=-1$ found in \textit{step~1}.
Integrating the function $f$ on $[0,r]$ using its expression~\eqref{eq:expr_f_k}, we get
\begin{equation*}
\begin{split}
v&=2\int_0^r f(x)\ dx
=-\frac{2}{|k|^{\alpha+1}}\int_0^r\int_{-1}^{kx+\frac{1}{2}}\frac{t(\frac{1}{2}-t)^\alpha}{\sqrt{1-t^2}}\ dtdx=\\
&=-\frac{2}{|k|^{\alpha+2}}\int_{-1}^\frac{1}{2}\int_t^{\frac{1}{2}}\frac{t(\frac{1}{2}-t)^\alpha}{\sqrt{1-t^2}}\ dxdt
=-\frac{2}{|k|^{\alpha+2}}\int_{-1}^\frac{1}{2}\frac{t(\frac{1}{2}-t)^{\alpha+1}}{\sqrt{1-t^2}}dt
\end{split}
\end{equation*}
applying Fubini's theorem, which immediately gives~\eqref{eq:th_expr_k}. By~\eqref{eq:mP_v} and~\eqref{eq:f'_kd}, we can also compute
\begin{equation*}
\begin{split}
\P_\alpha^x(E)&=2f(0)+4\int_0^r\sqrt{x^{2\alpha}+f'(x)^2}\ dx=\\
&=-\frac{2}{|k|^{\alpha+1}}\int_{-1}^{\frac{1}{2}}\frac{t(\frac{1}{2}-t)^\alpha}{\sqrt{1-t^2}}\ dt+4\int_0^r\sqrt{x^{2\alpha}+\frac{x^{2\alpha}(kx+\frac{1}{2})^2}{1-(kx+\frac{1}{2})^2}}\ dx=\\
&=-\frac{2}{|k|^{\alpha+1}}\int_{-1}^{\frac{1}{2}}\frac{t(\frac{1}{2}-t)^\alpha}{\sqrt{1-t^2}}\ dt+
\frac{4}{|k|^{\alpha+1}}\int_{-1}^{\frac{1}{2}}\frac{(\frac{1}{2}-t)^\alpha}{\sqrt{1-t^2}}\ dt=\\
&=\frac{2}{|k|^{\alpha+1}}\int_{-1}^{\frac{1}{2}}\frac{(\frac{1}{2}-t)^\alpha(2-t)}{\sqrt{1-t^2}}\ dt
\end{split}
\end{equation*}
using the information $kr+\frac{1}{2}=-1$ found in \textit{step~1}.
This concludes the proof.
\end{proof}

\input{dbeucl.tex}

\begin{example}[The Euclidean case]\label{ex:euclidean}
In the Euclidean case $\alpha=0$, the $x$-profile function defined in~\eqref{eq:th_expr_f_k} can be explicitly computed and we find 
\begin{equation*}
f(x)=\frac{1}{|k|}\sqrt{1-\left(kx+\frac{1}{2}\right)^2}=\sqrt{\frac{1}{k^2}-\left(x+\frac{1}{2k}\right)^2}
\end{equation*}
for all $x\in[0,r]$, where $k=-\sqrt{\frac{8\pi+3\sqrt{3}}{12v}}$ by~\eqref{eq:th_expr_k}. This is the profile function of a circle of radius $\frac{1}{|k|}$ and center $(0,\frac{1}{2|k|})$. In particular, we have $f'(0)=\frac{1}{\sqrt{3}}$ and the angle $\gamma=\arctan f'(0)$ is given by
\begin{equation*}
\gamma=\frac{\pi}{2}+\theta
=\frac{\pi}{2}+\arcsin\left(\frac{1}{2}\right)
=\frac{2\pi}{3},
\end{equation*}
see Figure~\ref{fig:dbeucl}. Thanks to Theorem~\ref{th:charact_reg_min}, up to Euclidean translations, the unique minimizer of problem~\eqref{eq:minpart} for $\alpha=0$ is the symmetric double bubble found in~\cite{FABHZ93}.
\end{example}

\input{grushin_bubble}

\begin{example}[The Grushin case]\label{ex:grushin_x}
In the Grushin case $\alpha=1$, the profile function defined in~\eqref{eq:th_expr_f_k} can be explicitly computed and we find 
\begin{equation*}
f(x)=\frac{1}{2k^2}\left(\frac{\pi}{2}+\arcsin\left(\frac{1}{2}+kx\right)+\left(\frac{1}{2}-kx\right)\sqrt{1-\left(kx+\frac{1}{2}\right)^2}\right)
\end{equation*}
for all $x\in[0,r]$, where $k=-\sqrt[3]{\frac{8\pi+9\sqrt{3}}{12v}}$ by~\eqref{eq:th_expr_k}, see Figure~\ref{fig:grushin_bubble}. Thanks to Theorem~\ref{th:charact_reg_min}, up to vertical translations, this is the $x$-profile function of the unique minimizer of problem~\eqref{eq:minpart} for $\alpha=1$.
\end{example}

\input{angles}

\subsubsection{Profile angle at the vertical interface}\label{sss:angle}

By Theorem~\ref{th:charact_reg_min}, for any $\alpha\ge0$ the $x$-profile function of the minimizer of problem~\eqref{eq:minpart} meets the vertical interface at an angle 
\begin{equation*}
\gamma=\frac{\pi}{2}+\theta, \qquad \theta=\arctan f'(0),
\end{equation*}
see Figure~\ref{fig:angles}. In the Euclidean case $\alpha=0$, we have $\theta=\frac{\pi}{6}$, as we found in Example~\ref{ex:euclidean} accordingly to the well-known regularity theory. In the Grushin case $\alpha>0$, instead, we have $\theta=0$, which means that the minimizer of problem~\eqref{eq:minpart} has a $C^1$-boundary consisting of two symmetric curves joining the vertical interface at two triple points with right angles. However, if we transform the Grushin plane $(\R^2,P_\alpha,\leb^2)$ into the Euclidean plane with weighted volume $(\R^2,P,\M_\alpha)$ using the maps defined in~\eqref{eq:def_Phi_Psi}, then the set $F=\Psi(E)$ has $\xi$-profile function $\hat{f}\colon[0,\hat{r}]\to[0,+\infty[$ given by
\begin{equation*}
\hat{f}(\xi)=f\left(((\alpha+1)\xi)^\frac{1}{\alpha+1}\right), \qquad \xi\in[0,\hat{r}],
\end{equation*}
where $\hat{r}=\frac{r^{\alpha+1}}{\alpha+1}$ and $f\colon[0,r]\to[0,+\infty[$ is the $x$-profile function of~$E$. An elementary computation shows that the profile angle at the interface in the transformed plane is given by $\hat{\gamma}=\frac{\pi}{2}+\hat{\theta}$, where
\begin{equation*}
\hat{\theta}=\arctan\hat{f}'(0)
=\arctan\left(\lim_{x\to0}\frac{f'(x)}{x^\alpha}\right)
=\arctan\left(\frac{1}{\sqrt{3}}\right)
=\frac{\pi}{6}.
\end{equation*}
In other words, the problem~\eqref{eq:minpart} reformulated in the transformed plane $(\R^2,P,\M_\alpha)$ has a unique minimizer consisting of two symmetric curves joining the vertical interface at two triple points with angles~$\frac{2\pi}{3}$.      

\section{The double bubble problem for horizontal interfaces}
\label{sec:horizontal}

%
%
\subsection{Existence of minimizers with horizontal interface}
\label{sec:existence_horizontal}

\subsubsection{Reduction to more symmetric sets}

In this section we prove the existence of solutions to problem~\eqref{eq:minpart_y}. The following result restricts the class of admissible sets to more symmetric ones.

\begin{theorem}[Reduction]\label{th:reduction_y}
Let $v>0$ and let $E\in\A^y(v)$ with $\P_\alpha^y(E)<+\infty$. There exists $\tilde{E}\in\A^y(v)\cap\S_x^*\cap\S_y$, bounded and $y$-transformed-convex, such that $\P_\alpha^y(\tilde{E})\le\P_\alpha^y(E)$. Moreover, in the case $\alpha>0$, $E\in\S_y$ is equivalent to a $y$-transformed-convex set if and only if $\P_\alpha^y(\tilde{E})=\P_\alpha^y(E)$. In addition, there exist $r\in]0,+\infty[$ and $g\in C([0,r])\cap\Lip_{loc}(]0,r[)$ such that
\begin{equation}\label{eq:reduction_profile_y}
\tilde{E}=\set*{(x,y)\in\R^2 : |x|\le g(|y|),\ |y|\le r}.
\end{equation}
\end{theorem}

\begin{proof}
The proof is similar to the one of Theorem~\ref{th:reduction} and we just sketch it. To avoid heavy notation, during the proof we will omit the $y$-superscript for the two sets $E^{\pm y}$ defined in~\eqref{eq:horizontal_pm}. 

\medskip

\textit{Step~1: $y$-symmetrization}. It is not restrictive to assume that $P_\alpha(E^+)\le P_\alpha(E^-)$ (otherwise, we can reflect $E$ with respect to the $x$-axis). We thus define
\begin{equation*}
E^1=\set*{(x,y)\in\R^2: (x,|y|)\in E^+}.
\end{equation*}
Clearly, $E^1\in\A^y(v)$ is $y$-symmetric. As in \textit{step~1} of the proof of Theorem~\ref{th:reduction}, one can prove that $\P_\alpha^x(E^1)\le\P_\alpha^x(E)$.

\medskip

\textit{Step~2: horizontal Steiner symmetrization}. Let $F^1=\Psi(E^1)$, where~$\Phi$ is as in~\eqref{eq:def_Phi_Psi}, and let $F^2$ be the Steiner symmetrization of $F^1$ in the $\xi$-direction. Precisely, recalling the notation introduced in~\eqref{eq:sections_notation},
\begin{equation*}
F^2=\set*{(\xi,t)\in\R^2 : |\xi|<\frac{\leb^1((F^1)^\eta_t)}{2}}.
\end{equation*}
Clearly, $F^{2,+}$ is the Steiner symmetrization of $F^{1,+}$ in the $\xi$-direction. Arguing as in \textit{step~2} of the proof of Theorem~\ref{th:reduction}, we have $\P^\eta(F^2)\le\P^\eta(F^1)$. Moreover, if $\P^\eta(F^2)<\P^\eta(F^1)$, then $F^1$ (or, equivalently, $F^{1,+}$) is not equivalent to an $\xi$-convex set. Arguing as in the proof of~\cite{MM04}*{Theorem~3.1}, we also have $\M_\alpha(F^{1,+})\le\M_\alpha(F^{2,+})$ with equality if and only if $F^{1,+}$ is equivalent to a $\xi$-convex set. We define $E^2=\Phi(F^2)$, where~$\Phi$ is as in~\eqref{eq:def_Phi_Psi}. By construction, we get that $E^2\in\A^x(v)\cap\S_x^*\cap\S_y$ satisfies $\P_\alpha^y(E^2)\le\P_\alpha^y(E^1)$ and $\leb^2(E^2)\ge\leb^2(E^1)$ with strict inequality if $E^1$ is not equivalent to a $x$-convex set.

\medskip

\textit{Step~3: horizontal rearrangement}. Let $F^2=\Psi(E^2)$,  where~$\Phi$ is as in~\eqref{eq:def_Phi_Psi}. Note that $F^{2,+}\subset\set*{\eta\ge0}$ is such that $\M_\alpha(F^{2,+})<+\infty$, $P(F^{2,+})<+\infty$ and $F^{2,+}\in\S_\xi^*$. As in \textit{step~3} of the proof of Theorem~\ref{th:reduction}, we have $\leb^2(F_2)<+\infty$. Up to perform a rotation of 90 degrees, we can thus apply Theorem~\ref{th:rearrang} to the set $F^{2,+}$. We define
\begin{equation*}
F^3=\set*{(\xi,\eta)\in\R^2 : (\xi,|\eta|)\in F^{3,+}}
\end{equation*}
where $F^{3,+}=(F^{2,+})^\bigstar$. We thus find that $F^{3,+}\in\S_\eta^*$ is $\eta$-convex, $\leb^2(F^{3,+})=\leb^2(F^{2,+})$ and
\begin{equation*}
\min\set*{P(F^{2,+})-P(F^{3,+}),P(F^{2,+};\set{\eta>0})-P(F^{3,+};\set{\eta>0})}\ge0,
\end{equation*}
with equality if and only if $F^{2,+}$ is equivalent to a $\eta$-convex set. Therefore $\P^\eta(F^3)\le\P^\eta(F^2)$, with equality if and only if $F^{2,+}$ is equivalent to a $\eta$-convex set. By Tonelli's theorem, we also have that $\M_\alpha(F^{3,+})=\M_\alpha(F^{2,+})$. We thus set $E^3=\Phi(F^3)$. We have that $E^3\in\A^x(v)\cap\S_x^*\cap\S_y$ is such that $E^{3,+}$ is $y$-convex and satisfies $\leb^2(E^3)=\leb^2(E^2)$ and $\P_\alpha^y(E^3)\le\P_\alpha^y(E^2)$, with strict inequality if and only if $E^{2,+}$ is not equivalent to a $y$-convex set. 

We can now conclude as in \textit{step~4} and \textit{step~5} of the proof of Theorem~\ref{th:reduction} with minor modifications. We leave the details to the reader.
\end{proof}

\subsubsection{Existence of minimal double bubbles with horizontal interface}

We are now ready to prove the existence of minimizers to problem~\eqref{eq:minpart_y}.

\begin{theorem}[Existence of minimizers with horizontal interface]\label{th:existence_y}
Let $\alpha\ge0$ and fix $v>0$. There exists a solution $E^*\in\A_y(v)\cap\S^*_x\cap\S_y$, bounded and $y$-transformed-convex, to the minimal partition problem~\eqref{eq:minpart_y}. Moreover, there exist $r\in]0,+\infty[$ and $g\in C([0,r])\cap\Lip_{loc}(]0,r[)$ such that
\begin{equation*}
E^*=\set*{(x,y)\in\R^2 : |x|\le g(|y|),\ |y|\le r}.
\end{equation*}
\end{theorem}

\begin{proof}
The proof is similar to the one of Theorem~\ref{th:existence}. By Theorem~\ref{th:reduction_y}, the class of admissible sets can be restricted to
\begin{equation*}
\mathcal{B}^y(v)=\set*{E\in\A^y(v) : E\in\S_x^*\cap\S_y\ \text{bounded, $y$-transformed-convex, as in~\eqref{eq:reduction_profile_y}}}.
\end{equation*}
By the isoperimetric inequality~\eqref{eq:isop_ineq}, for any $E\in\mathcal{B}^y(v)$ we have that $\P_\alpha^y(E)\le k$ for some constant $k=k(\alpha,v)\in]0,+\infty[$ depending only on~$\alpha\ge0$ and~$v>0$. 

We claim that any $E\in\mathcal{B}^y(v)$ is contained in the rectangle $[-a,a]\times[-b,b]$ for some $a,b\in]0,+\infty[$ depending only on $\alpha,v>0$. Fix $E\in\mathcal{B}^y(v)$. By Theorem~\eqref{th:reduction_y}, there exist $r_E\in]0,+\infty[$ and $g_E\in C([0,r_E])\cap\Lip_{loc}(]0,r_E[)$ such that 
\begin{equation*}
E=\set*{(x,y)\in\R^2 : |x|\le g_E(|y|),\ |y|\le r_E}.
\end{equation*}
Thus, by the representation formula~\eqref{eq:perim_from_y_to_x}, we get
\begin{equation*}
k\ge P_\alpha(E^{+y})
\ge 2\int_0^{r_E}\sqrt{1+g_E(y)^{2\alpha}g_E'(y)^2}\ dy
\ge2r_E
\end{equation*}
and thus $2b\le k$. In particular, the convex set $F^{+\eta}=\Phi(E^{+y})$ is bounded and contained in $\R\times[0,\bar{b}]$ for some $\bar{b}$ depending on~$b$ and~$\alpha$. Let
\begin{equation*}
\bar{a}:=\max\set*{\leb^1(F^\xi_t) : t\in[0,\bar{b}]}.
\end{equation*}
Then, by convexity, we get $k\ge P(F^{+\eta})\ge\sqrt{2\bar{a}^2+\bar{b}^2}$, which immediately implies that~$a$ depends only on~$\alpha,v>0$. The proof can now be concluded similarly to the one of Theorem~\eqref{th:existence}, with minor modifications. We leave the details to the reader.   
\end{proof}

\subsection{Characterization of minimizers with horizontal interface}
\label{sec:characterization_horizontal}

\subsubsection{Regular minimizers}
\label{sss:regular_y}
In this section, we solve the minimal partition problem~\eqref{eq:minpart_y}. In Theorem~\ref{th:existence_y}, we proved the existence of a minimizer $E\in\A^y(v)\cap\S_x\cap\S_y^*$ that is bounded, $y$-transformed-convex and of the form
\begin{equation}\label{eq:regular_min_y}
E=\set*{(x,y)\in\R^2 : |x|\le g(|y|),\ |y|\le r}
\end{equation}
for some \emph{$y$-profile function} $g\in C([0,r])\cap\Lip_{loc}(]0,r[)$ with $r\in]0,+\infty[$. We call such a  minimizer a \emph{$y$-regular minimizer}. In case of a $y$-regular minimizer~$E$, the functional~\eqref{eq:Pa_y} takes the form
\begin{equation}\label{eq:mP_v_y}
\P_\alpha^y(E)=2\int_0^{g(0)} x^\alpha\ dx+4\int_0^{g(r)} x^\alpha\ dx+4\int_0^r\sqrt{1+g(y)^{2\alpha}g'(y)^2}\ dy.
\end{equation}

By Theorem~\ref{th:reduction_y}, any minimizer $E\in\A^y(v)\cap\S_y$ of problem~\eqref{eq:minpart_y} is a $y$-regular minimizer. By Lemma~\ref{lemma:min_are_y-simm} below, if we prove that there exists a unique $y$-symmetric minimizer up to vertical translations, then this must be the unique minimizer of problem~\eqref{eq:minpart_y} up to vertical translations. The proof of Lemma~\ref{lemma:min_are_y-simm} is analogous to the one of Lemma~\ref{lemma:min_are_x-simm} and we omit it.

\begin{lemma}\label{lemma:min_are_y-simm}
Assume that problem~\eqref{eq:minpart_y} has a unique $y$-symmetric minimizer $E_0\in\A^y(v)\cap\S_y$ up to vertical translations. Then $E_0$ is the unique minimizer of problem~\eqref{eq:minpart_y} up to vertical translations. 
\end{lemma}

\subsubsection{Characterization and examples}

We are now ready for the main result of this section. In Theorem~\ref{th:charact_reg_min_y} below, we prove that, given $\alpha\ge0$ and $v>0$, the $y$-regular minimizer of problem~\eqref{eq:minpart_y} is unique and has smooth boundary far from the $x$-axis. By Lemma~\ref{lemma:min_are_y-simm} and and Section \ref{sss:regular_y}, this is the unique minimizer of problem~\eqref{eq:minpart_y} up to vertical translations. 

\begin{theorem}[Characterization of minimal double bubbles with horizontal interface]\label{th:charact_reg_min_y}
Let $\alpha\ge0$, $v>0$ and let $E\in\mathcal A^y(v)\cap\S^*_x\cap\S_y$ be a $y$-regular minimizer of problem~\eqref{eq:minpart_y} for some $y$-profile function $g\in C([0,r])\cap\Lip_{loc}(]0,r[)$ with $r\in]0,+\infty[$. Then $E$ is unique and is given by
\begin{equation*}
E=\delta_{\frac{1}{h}}\left(\set*{(x,y)\in\R^2 : \left(x,|y|-\phi_\alpha\left(\tfrac{\sqrt{3}}{2}\right)\right)\in E_\alpha}\right),
\end{equation*}
where $\phi_\alpha\colon[0,1]\to[0,+\infty[$ is the isoperimetric profile defined in~\eqref{eq:isop_profile} of the Grushin isoperimetric set~$E_\alpha$ recalled in~\eqref{eq:isop_set} and 
\begin{equation}
\label{eq:th_expr_k_y}
h=\left[\frac{1}{v}\left(\leb^2(E_\alpha)-2\int_0^\frac{\sqrt{3}}{2}\frac{t^{\alpha+2}}{\sqrt{1-t^2}}\ dt\right)\right]^\frac{1}{\alpha+2}.
\end{equation}
The $y$-profile function of $E$ is given by
\begin{equation}\label{eq:profile_y}
g(y)=\tfrac{1}{h}\,\phi_\alpha^{-1}\left(\abs*{h^{\alpha+1}y-\phi_\alpha\left(\tfrac{\sqrt{3}}{2}\right)}\right)
\end{equation}
for all $y\in[0,r]$. In particular, $g\in C^\infty(]0,r[)$ and $r=h^{-(\alpha+1)}\left(r_\alpha+\phi_\alpha\left(\frac{\sqrt{3}}{2}\right)\right)$, where $r_\alpha=\phi_\alpha(0)$. Moreover, $g$ satisfies $g(0)>0$, $g(r)=0$, $g'(r)=-\infty$ and \mbox{$g(0)^\alpha g'(0)=\frac{1}{\sqrt{3}}$}. Finally, the minimum of problem~\eqref{eq:minpart_y} is given by
\begin{equation}\label{eq:th_expr_P_alpha_y}
\P_\alpha^y(E)=\frac{2}{h^{\alpha+1}}\left[P_\alpha(E_\alpha)+\int_0^\frac{\sqrt{3}}{2}t^\alpha\left(1-\frac{2}{\sqrt{1-t^2}}\right) dt\right].
\end{equation}
\end{theorem}

\begin{proof}
We split the proof in several steps.

\medskip

\textit{Step~1: differential equation for the $y$-profile}.
We perform a first variation argument. Let $\psi\in C_c^\infty(]0,r[)$ be such that $\int_0^r\psi\ dy=0$. For $\eps\in\R$ small, consider the set 
\begin{equation*}
E_\eps=\set*{(x,y)\in\R^2 : |x|\le g(|y|)+\eps\psi(|y|),\ |y|\le r}.
\end{equation*}
Note that $E_\eps\in\A^y(v)$ for all $\eps\in\R$ small, since $g(y)>0$ for every $y\in]0,r[$ by definition.
Then, by~\eqref{eq:mP_v_y} and the minimality of $E$, after an integration by parts we find
\begin{equation*}
0=\frac{d\P_\alpha^y(E_\eps)}{d\eps}\bigg|_{\eps=0}
=4\int_0^r\left\{\frac{\alpha g(y)^{2\alpha-1}g'(y)^2}{\sqrt{1+g(y)^{2\alpha}g'(y)^2}}-\frac{d}{dy}\left(\frac{g(y)^{2\alpha}g'(y)}{\sqrt{1+g(y)^{2\alpha}g'(y)^2}}\right)\right\}\psi(y)\ dy.
\end{equation*}
Thus, by the fundamental lemma of the Calculus of Variations, there exists a constant $c\in\R$ such that
\begin{equation}\label{eq:def_k_y}
\frac{\alpha g(y)^{2\alpha-1}g'(y)^2}{\sqrt{1+g(y)^{2\alpha}g'(y)^2}}-\frac{d}{dy}\left(\frac{g(y)^{2\alpha}g'(y)}{\sqrt{1+g(y)^{2\alpha}g'(y)^2}}\right)=c
\end{equation}
for all $y\in]0,r[$. In addition, by the regularity theory of $\Lambda$-minimizers of perimeter, the boundary $\de E$ is smooth far from the $y$-axis. Therefore we must have $g\in C^\infty(]0,r[)$.

\medskip

\textit{Step~2: $x$-profile near the point $(x,y)=(0,r)$}. Since $E$ is $y$-transformed-convex, there exist $\delta,\eta>0$ and $f\in C([-\delta,\delta])\cap\Lip_{loc}(]-\delta,\delta[)$ such that
\begin{equation*}
\de E\cap\set*{(x,y)\in\R^2 : x\in[-\delta,\delta],\ y\in[r-\eta,r+\eta]}=\set*{(x,f(x))\in\R^2 : x\in[-\delta,\delta]}.
\end{equation*} 
Performing a first variation argument and arguing as in the proof of~\cite{MM04}*{Theorem~3.2}, we find a constant $b<0$ such that
\begin{equation}
f'(x)=\sgn(x)\,\frac{b|x|^{\alpha+1}}{\sqrt{1-b^2x^2}}
\end{equation} 
for all $x\in]-\delta,\delta[$. In particular, up to a dilation and a vertical translation, the function $f\colon[-\delta,\delta]\to[0,+\infty)$ conincides with the Grushin isoperimetric profile recalled in~\eqref{eq:isop_profile}.
For $x\in[0,\delta]$ and $y\in[f(\delta),r]$ we have $y=f(x)$ if and only if $ x=g(y)$. As a consequence, we must have
\begin{equation}\label{eq:g(r)_and_g'(r)}
g(r)=0,
\qquad
g'(r):=\lim_{y\to r}g'(y)=-\infty.
\end{equation}
Moreover, we have
\begin{equation}\label{eq:g'_north_pole}
g'(y)=(f^{-1})'(y)=\frac{1}{f'(g(y))}=\frac{\sqrt{1-b^2g(y)^2}}{b\,g(y)^{\alpha+1}}
\end{equation}
for all $y\in]f(\delta),r[$. Inserting~\eqref{eq:g'_north_pole} in~\eqref{eq:def_k_y}, we get $b=c$.

\medskip

\textit{Step~3: $g$ has a strict maximum in $]0,r[$}. 
Define $G\colon]0,r[\to\R$ by setting
\begin{equation}\label{eq:def_G}
G(y)=\frac{g(y)^{2\alpha}g'(y)}{\sqrt{1+g(y)^{2\alpha}g'(y)^2}}
\end{equation}
for all $y\in]0,r[$. The differential equation in~\eqref{eq:def_k_y} can be rewritten as
\begin{equation}\label{eq:ODE_G}
G'(y)-\alpha\frac{g'(y)}{g(y)}\,G(y)=-c
\end{equation}
for all $y\in]0,r[$. Fix $y_0\in]0,r[$. Then, integrating~\eqref{eq:ODE_G}, we have
\begin{equation}\label{eq:find_G}
G(y)=g(y)^\alpha\left(\frac{G(y_0)}{g(y_0)^\alpha}-c\int_{y_0}^y\frac{dt}{g(t)^\alpha}\right)
\end{equation}
for all $y\in]0,r[$. Combining~\eqref{eq:def_G} and~\eqref{eq:find_G}, we get
\begin{equation}\label{eq:sign_of_g'}
\frac{g(y)^{\alpha}g'(y)}{\sqrt{1+g(y)^{2\alpha}g'(y)^2}}=\frac{G(y_0)}{g(y_0)^\alpha}-c\int_{y_0}^y\frac{dt}{g(t)^\alpha}
\end{equation}
for all $y\in]0,r[$. Thus $g'$ can change sign at most one time on the interval $]0,r[$, because the function appearing on the right-hand side of~\eqref{eq:sign_of_g'} is strictly monotone since $g(y)>0$ for all $y\in]0,r[$. 

By contradiction, assume that $g'$ does not change sign on $]0,r[$, so that, by~\eqref{eq:g(r)_and_g'(r)}, $g'(y)<0$ for all $y\in]0,r[$. Therefore, by the implicit function theorem, we can extend the function $f\colon[-\delta,\delta]\to[0,+\infty[$ to the interval $[-g(0),g(0)]$. Note that we must have that $g(0)>0$, since $g(r)=0$ by~\eqref{eq:g(r)_and_g'(r)} and since we are assuming that $g'<0$ on $]0,r[$. We can thus repeat the argument contained in \textit{step~2} for all $x\in[0,g(0)]$ and $y\in[0,r]$ and deduce that
\begin{equation}\label{eq:g'_contradiction}
g'(0):=\lim_{y\to0}g'(y)
=\lim_{y\to0}\frac{\sqrt{1-c^2g(y)^2}}{c\,g(y)^{\alpha+1}}=\frac{\sqrt{1-c^2g(0)^2}}{c\,g(0)^{\alpha+1}}\le0
\end{equation}
passing to the limit in~\eqref{eq:g'_north_pole} as $y\to0$, because $c<0$ as we found in \textit{step~2}. 

We perform a first variation argument. Let $\psi\in C_c^\infty([0,r))$ be such that $\int_0^r\psi\ dy=0$. For $\eps\in\R$ small, consider the set 
\begin{equation*}
E_\eps=\set*{(x,y)\in\R^2 : |x|\le g(|y|)+\eps\psi(|y|),\ |y|\le r}.
\end{equation*}
Note that $E_\eps\in\A^y(v)$ for all $\eps\in\R$ small, since $g(y)>0$ for every $y\in[0,r[$.
Then, by~\eqref{eq:mP_v_y} and the minimality of $E$, we find
\begin{equation*}
0=\frac{d\P_\alpha^y(E_\eps)}{d\eps}\bigg|_{\eps=0}
=2\psi(0)g(0)^\alpha\left(1-\frac{2g(0)^\alpha g'(0)}{\sqrt{1+g(0)^{2\alpha}g'(0)^2}}\right).
\end{equation*}
By the arbitrariness of $\psi$, we deduce that
\begin{equation}\label{eq:new_d}
\frac{g(0)^\alpha g'(0)}{\sqrt{1+g(0)^{2\alpha}g'(0)^2}}=\frac{1}{2}.
\end{equation}
Therefore we must have $g'(0)>0$, contradicting~\eqref{eq:g'_contradiction}.

We conclude that $g'$ must change sign exactly one time on the interval $]0,r[$. So there must be a point $\hat{y}\in]0,r[$ such that $g'(\hat{y})=0$, $g'(y)>0$ for $y\in]0,\hat{y}[$ and $g'(y)<0$ for $y\in]\hat{y},r[$. In particular, $g$ has a strict maximum point at $y=\hat{y}$.

\medskip

\textit{Step~4: symmetry with respect to $\hat{y}$ and proof of $g(0)>0$}. 
We prove that $g$ is symmetric with respect to $\hat{y}$. Indeed, define $\hat{g}(y):=g(2\hat{y}-y)$ for all $y\in[2\hat{y}-r,2\hat{y}]\cap[0,r]$. Then $\hat{g}(\hat{y})=g(\hat{y})$ and $\hat{g}'(\hat{y})=-g'(\hat{y})=0$. A direct computation shows that $\hat{g}$ solves the differential equation found in~\eqref{eq:def_k_y} and the conclusion follows by the Cauchy-Lipschitz theorem. This proves that there is $\hat{\lambda}>0$ such that 
\begin{equation}\label{eq:shape_y*_lambda}
E=\set*{(x,y)\in\R^2 : (x,|y|-\hat{y})\in \delta_{\hat{\lambda}}(E_\alpha)},
\end{equation}
where $E_\alpha$ is the Grushin isoperimetric set recalled in~\eqref{eq:isop_set}. 

We now prove that $g(0)>0$. Assume $g(0)=0$ by contradiction. For any $\eps>0$ sufficiently small, we define $g_\eps\colon[0,r]\to[0,+\infty[$ by setting 
\begin{equation*}
g_\eps(y)=
\begin{cases}
g(y+\eps) & x\in[0,r-\eps[\\
0 & x\in[r-\eps,r].
\end{cases}
\end{equation*}
We then consider the set 
\begin{equation*}
E_\eps:=\set*{(x,y)\in\R^2 : |x|\le g_\eps(|y|),\ |y|\le r}.
\end{equation*}
Note that
\begin{equation*}
\P_\alpha^y(E_\eps)=\P_\alpha^y(E)+\frac{2g(\eps)^{\alpha+1}}{\alpha+1}-4\int_\eps^r\sqrt{1+g(y)^{2\alpha}g'(y)^2}\ dy
\end{equation*}
and
\begin{equation*}
\leb^2(E_\eps)=\leb^2(E)-4\int_0^\eps g(y)\ dy.
\end{equation*}
We thus define
\begin{equation*}
\lambda_\eps=\left(\frac{\leb^2(E)}{\leb^2(E_\eps)}\right)^\frac{1}{\alpha+2}>1
\qquad\text{and}\qquad
F_\eps=\delta_{\lambda_\eps}(E_\eps).
\end{equation*}
Then $F_\eps\in\A^y(v)$ and $\P_\alpha^y(F_\eps)=\lambda_\eps^{\alpha+1}\P_\alpha^y(E_\eps)$. We claim that $\P_\alpha^y(F_\eps)<\P_\alpha^y(E)$ for any $\eps>0$ sufficiently small. A direct computation gives
\begin{equation*}
\frac{d\P_\alpha^y(F_\eps)}{d\eps}\bigg|_{\eps=0}=-4
\end{equation*}
and the claim follows from the Taylor's expansion of the function $\eps\mapsto\P_\alpha^y(F_\eps)$. But this contradicts the minimality of $E$.

We conclude that $g(0)>0$. Therefore, we can repeat the first variation argument presented in \textit{step~3} and we have~\eqref{eq:new_d}. Simplifying, we get the necessary condition
\begin{equation}\label{eq:angle_constraint_y}
g(0)^\alpha g'(0)=\frac{1}{\sqrt{3}}.
\end{equation}  

\medskip

\textit{Step~5: characterization of the profile}. 
We introduce a family of sets $(F_\tau)_{\tau\in[0,1]}$ constructed as follows.
Let $\tau\in[0,1]$ and define
\begin{equation*}
E_\tau=\set*{(x,y)\in\R^2 : (x,|y|-\phi_\alpha(\tau))\in E_\alpha},
\end{equation*}
where $E_\alpha\subset\R^2$ is the Grushin isoperimetric set recalled in~\eqref{eq:isop_set}. Then we have
\begin{equation}
\label{eq:per_tau}
\P_\alpha^y(E_\tau)
=2P_\alpha(E_\alpha)+2\int_0^\tau x^\alpha\ dx-4\int_0^\tau\sqrt{x^{2\alpha}+\phi_\alpha'(x)^2}\ dx.
\end{equation}
and
\begin{equation}\label{eq:vol_tau}
\leb^2(E_\tau)
=2\leb^2(E_\alpha)-4\int_0^\tau\phi_\alpha(x)\ dx+4\tau\phi_\alpha(\tau).
\end{equation}
We thus define 
\begin{equation}\label{eq:def_F_tau}
F_\tau=\delta_{\lambda_\tau}(E_\tau)=\delta_{\lambda_\tau}\left(\set*{(x,y)\in\R^2 : (x,|y|-\phi_\alpha(\tau))\in E_\alpha}\right),
\end{equation}
where
\begin{equation}
\label{eq:lambda_tau}
\lambda_\tau=\left(\frac{2v}{\leb^2(E_\tau)}\right)^\frac{1}{\alpha+2}
\end{equation}
for all $\tau\in[0,1]$. By construction, for any $\tau\in[0,1]$ there exist $r_\tau>0$ and a smooth function $g_\tau\colon[0,r_\tau]\to[0,+\infty[$ such that
\begin{equation*}
F_\tau=\set*{(x,y)\in\R^2 : |x|\le g_\tau(|y|),\ |y|\le r_\tau}.
\end{equation*}
A straightforward computation gives $r_\tau=\lambda_\tau^{\alpha+1}(\phi_\alpha(0)+\phi_\alpha(\tau))$ and
\begin{equation}\label{eq:def_g_tau}
g_\tau(y)=\lambda_\tau\phi_\alpha^{-1}\left(\abs*{\lambda_\tau^{-(\alpha+1)}y-\phi_\alpha(\tau)}\right)
\end{equation}
for all $y\in[0,r_\tau]$. 

We claim that any minimizer of problem~\eqref{eq:minpart_y} of the form~\eqref{eq:shape_y*_lambda} can be written in the form~\eqref{eq:def_F_tau} for some $\tau\in[0,1]$ depending on $\hat{y},\hat{\lambda}>0$. Indeed, the set in~\eqref{eq:shape_y*_lambda} can be rewritten as
\begin{equation*}
E=\delta_{\hat{\lambda}}\left(\set*{(x,y)\in\R^2 : \left(x,|y|-\hat{y}\hat{\lambda}^{-(\alpha+1)}\right)\in E_\alpha}\right).
\end{equation*} 
Therefore, recalling that $E$ is a $y$-regular minimizer as in~\eqref{eq:regular_min_y}, we must have
\begin{equation*}
\set*{(x,y)\in\R^2 : \hat{\lambda}|x|\le g\left(\hat{\lambda}^{(\alpha+1)}|y|\right),\ |y|\le r}
=\set*{(x,y)\in\R^2 : \left(x,|y|-\hat{y}\hat{\lambda}^{-(\alpha+1)}\right)\in E_\alpha}
\end{equation*}
and so $(\hat{\lambda}^{-1}g(\hat{\lambda}^{(\alpha+1)}y),y-\hat{y}\hat{\lambda}^{-(\alpha+1)})\in\de E_\alpha$ for all $y\in[0,r]$. Thus we must have that $\abs{y-\hat{y}\hat{\lambda}^{-(\alpha+1)}}=\phi_\alpha(\hat{\lambda}^{-1}g(\hat{\lambda}^{(\alpha+1)}y))$ for all $y\in[0,r]$, from which we deduce that $g(y)=\hat{\lambda}\phi_\alpha^{-1}\left(\lambda^{-\alpha+1}\abs*{y-\hat{y}}\right)$ for all $y\in[0,r]$. Comparing this function with the one in~\eqref{eq:def_g_tau}, we conclude that $\tau=\phi_\alpha^{-1}(\hat{y}\hat{\lambda}^{-(\alpha+1)})$ and the claim follows.

We now look for the values of $\tau\in[0,1]$ such that the $y$-profile $g_\tau$ satisfies the necessary condition~\eqref{eq:angle_constraint_y}. Starting from~\eqref{eq:def_g_tau}, a direct computation gives
\begin{equation*}
g_\tau(0)=\tau\lambda_\tau 
\qquad\text{and}\qquad
g_\tau'(0)=\frac{\sqrt{1-\tau^2}}{\lambda_\tau^\alpha\,\tau^{\alpha+1}} 
\end{equation*} 
for all $\tau\in[0,1]$. Inserting these values in~\eqref{eq:angle_constraint_y} and simplifying, we obtain $\tau=\frac{\sqrt{3}}{2}$. Setting $h_\tau=\lambda^{-1}_\tau$, formulas~\eqref{eq:profile_y} and~\eqref{eq:th_expr_P_alpha_y} follow by~\eqref{eq:per_tau} and~\eqref{eq:def_g_tau} respectively. Note that, by~\eqref{eq:vol_tau} and~\eqref{eq:lambda_tau}, we get
\begin{align*}
h_\tau^{\alpha+2}v 
&=\leb^2(E_\alpha)-2\int_0^{\frac{\sqrt{3}}{2}}\int_x^1\frac{t^{\alpha+1}}{\sqrt{1-t^2}}\,dt\,dx+\sqrt{3}\,\varphi_\alpha\left(\tfrac{\sqrt{3}}{2}\right)\\
&=\leb^2(E_\alpha)-2\int_0^{\frac{\sqrt{3}}{2}}\int_0^t\frac{t^{\alpha+1}}{\sqrt{1-t^2}}\,dx\,dt-\sqrt{3}\int_{\frac{\sqrt{3}}{2}}^1\frac{t^{\alpha+1}}{\sqrt{1-t^2}}\,dt+\sqrt{3}\,\varphi_\alpha\left(\tfrac{\sqrt{3}}{2}\right)\\
&=\leb^2(E_\alpha)-2\int_0^\frac{\sqrt{3}}{2}\frac{t^{\alpha+2}}{\sqrt{1-t^2}}\, dt
\end{align*}
and~\eqref{eq:th_expr_k_y} follows immediately. This concludes the proof.
\end{proof}

\input{dbeucl_y}

\begin{example}[The Euclidean case]\label{ex:euclidean_y}
In the Euclidean case $\alpha=0$, the $y$-profile function defined in~\eqref{eq:th_expr_f_k} can be explicitly computed. Recalling~\eqref{eq:isop_profile}, we have
\begin{equation*}
\phi_0(x)=\int_{\arcsin x}^\frac{\pi}{2}\sin(t)\ dt=\cos\arcsin x=\sqrt{1-x^2}
\end{equation*}
for all $x\in[0,1]$. Thus
\begin{equation*}
g(y)=\frac{1}{h}\sqrt{1-\left(h y-\frac{1}{2}\right)^2}=\sqrt{\frac{1}{h^2}-\left(y-\frac{1}{2h}\right)^2}
\end{equation*}
for all $y\in[0,r]$, where $h=\sqrt{\frac{8\pi+3\sqrt{3}}{12v}}$ and $r=\frac{3}{2h}$ by~\eqref{eq:th_expr_k_y}. This is the profile function of a circle of radius $\frac{1}{h}$ and center $(\frac{1}{2h},0)$. In particular, we have $g'(0)=\frac{1}{\sqrt{3}}$ and the angle $\gamma=\frac{\pi}{2}+\theta$, with $\theta=\arctan g'(0)$, is given by
\begin{equation*}
\gamma
=\frac{\pi}{2}+\arctan\left(\frac{1}{\sqrt{3}}\right)
=\frac{2\pi}{3},
\end{equation*}
see Figure~\ref{fig:dbeucl_y}. Thanks to Theorem~\ref{th:charact_reg_min_y}, up to Euclidean translations, the unique minimizer of problem~\eqref{eq:minpart_y} for $\alpha=0$ is the symmetric Standard Double Bubble found in~\cite{FABHZ93}.
\end{example}

\input{isop_bubble}

\begin{example}[The Grushin case]\label{ex:grushin_y}
In the Grushin case $\alpha>0$, the profile function defined in~\eqref{eq:profile_y} cannot be explicitly computed. For $\alpha=1$, recalling~\eqref{eq:isop_profile}, we have 
\begin{equation*}
\phi_1(x)=\int_{\arcsin x}^\frac{\pi}{2}\sin^2(t)\ dt=\frac{1}{2}\left(\arccos x+x\sqrt{1-x^2}\right)
\end{equation*}
for all $x\in[0,1]$ and so $\phi_1(\frac{\sqrt{3}}{2})=\frac{2\pi+3\sqrt{3}}{24}$. In addition, we can explicitly compute 
\begin{equation}\label{eq:beta1}
h=\left[\frac{1}{v}\left(\leb^2(E_1)-2\int_0^\frac{\sqrt{3}}{2}\frac{t^3}{\sqrt{1-t^2}}\ dt\right)\right]^\frac{1}{3}
=\sqrt[3]{\frac{9}{4v}}.
\end{equation}
Thus
\begin{equation*}
g(y)=\tfrac{1}{h}\phi_1^{-1}\left(\abs*{h^2 y-\tfrac{2\pi+3\sqrt{3}}{24}}\right)
\end{equation*}
for all $y\in[0,r]$, where $r=\frac{8\pi+3\sqrt{3}}{24h^2}$. By~\eqref{eq:angle_constraint_y} we have $g'(0)=\frac{1}{g(0)\sqrt{3}}=\sqrt[3]{\frac{2}{3v}}$, so that the angle $\gamma=\frac{\pi}{2}+\theta$, with $\theta=\arctan g'(0)$, is given by
\begin{equation*}
\gamma=\frac{\pi}{2}+\arctan\left(\sqrt[3]{\frac{2}{3v}}\right),
\end{equation*}
see Figure~\ref{fig:isop_bubble}. In particular, $\gamma=\frac{2\pi}{3}$ if and only if $v=\frac{2}{\sqrt{3}}$.
\end{example} 

\subsubsection{Profile angle at the horizontal interface}
\label{sss:angle_y}

\input{angles_y}

By Theorem~\ref{th:charact_reg_min_y}, for any $\alpha\ge0$ the $y$-profile function of the minimizer of problem~\eqref{eq:minpart_y} meets the horizontal interface at an angle 
\begin{equation*}
\gamma=\frac{\pi}{2}+\theta, \qquad \theta=\arctan g'(0),
\end{equation*}
see Figure~\ref{fig:angles_y}. In the Euclidean case $\alpha=0$, we have $\theta=\frac{\pi}{6}$, as we found in Example~\ref{ex:euclidean_y} accordingly to the well-known regularity theory. In the Grushin case $\alpha>0$, instead, we have 
\begin{equation*}
\theta
=\arctan\left(\frac{1}{g(0)^\alpha\sqrt{3}}\right)
=\arctan\left[\left(\frac{2h}{\sqrt{3}}\right)^\alpha\cdot\frac{1}{\sqrt{3}}\right].
\end{equation*}
However, if we transform the Grushin plane $(\R^2,P_\alpha,\leb^2)$ into the Euclidean plane with weighted volume $(\R^2,P,\M_\alpha)$ using the maps defined in~\eqref{eq:def_Phi_Psi}, then the set $F=\Psi(E)$ has $\eta$-profile function $\hat{g}\colon[0,r]\to[0,+\infty[$ given by
\begin{equation*}
\hat{g}(\eta)=\frac{g(\eta)^{\alpha+1}}{\alpha+1}, \qquad \eta\in[0,r],
\end{equation*}
where $g\colon[0,r]\to[0,+\infty[$ is the $y$-profile function of~$E$. An elementary computation shows that the profile angle at the interface in the transformed plane is given by $\hat{\gamma}=\frac{\pi}{2}+\hat{\theta}$, where
\begin{equation*}
\hat{\theta}=\arctan\hat{g}'(0)
=\arctan\left(g(0)^\alpha g'(0)\right)
=\arctan\left(\frac{1}{\sqrt{3}}\right)
=\frac{\pi}{6}.
\end{equation*}
In other words, the problem~\eqref{eq:minpart} reformulated in the transformed plane $(\R^2,P,\M_\alpha)$ has a unique minimizer consisting of two symmetric curves joining the vertical interface at two triple points with angles~$\frac{2\pi}{3}$.

\begin{remark}[Comparison of the two minimal bubbles for $\alpha=1$]\label{rem:comparison}
Having in mind the general problem~\eqref{eq:intro_bubble_cluster_problem} for $m=2$, we want to understand which of the two double bubbles characterized in Theorems~\ref{th:charact_reg_min} and~\ref{th:charact_reg_min_y} may be a candidate solution to the double bubble problem for equal areas in the Grushin plane $(\R^2,P_\alpha,\leb^2)$. Since the expressions of the values of problems~\eqref{eq:minpart} and~\eqref{eq:minpart_y} given in~\eqref{eq:th_expr_P_alpha} and~\eqref{eq:th_expr_P_alpha_y} are not easily comparable for an arbitrary $\alpha>0$, we restrict our analysis to the case $\alpha=1$. 

Besides its computational manageability, the case $\alpha=1$ is of particular interest since it is connected with the Heisenberg group $\mathbb{H}^1$. This is the framework of the famous Pansu's conjecture about the shape of isoperimetric sets, see \cite{P82}, which is still unsolved. Pansu's set can be obtained by rotating the set $E_\alpha$ in \eqref{eq:isop_set} for $\alpha=1$ around the vertical axis in~$\R^3$. Our analysis might give some insights on the candidate solutions to the double bubble problem in $\mathbb H^1$.

Now let $E_x$ and $E_y$ be the minimal bubbles given by Theorems~\ref{th:charact_reg_min} and~\ref{th:charact_reg_min_y} respectively. From \eqref{eq:th_expr_P_alpha} and  Example~\ref{ex:grushin_x} we get 
\begin{equation*}
\mathcal P_1(E_x)=\mathcal P^x_1(E_x)
=(9\sqrt{3}+8\pi)^\frac{1}{3}\Big(\frac{3}{2}\Big)^\frac{2}{3}v^\frac{2}{3},
\end{equation*}
while from \eqref{eq:th_expr_P_alpha_y} and Example~\ref{ex:grushin_y} we get 
\[
\mathcal P_1(E_y)=\mathcal P^y_1(E_y)=3\Big(\frac{3}{2}\Big)^\frac{2}{3}v^\frac{2}{3}.
\]
Thus $\P_1(E_y)<\P_1(E_x)$. Motivated by this comparison and by the fact that~$E_y$ is obtained by translating and dilating the Grushin isoperimetric set, we conjecture that~$E_y$ may be the solution of the double bubble problem for equal areas in the Grushin plane. 
\end{remark}


\end{document}

%% file: rearrang_polyhed.tex
\begin{figure}
\begin{tikzpicture}[scale=.8]

\draw [thick,->] (-0.5,0) -- (8,0); 
\draw [thick,->] (0,-0.5) -- (0,4.5); 
\draw (8,0) node [right] {\small $x$};
\draw (0,4.5) node [above] {\small $y$};

\draw (1,0) -- (1.5,2) -- (2,3) -- (2.5,2) -- (3,0) -- cycle;
\fill [gray,opacity=0.3] (1,0) -- (1.5,2) -- (2,3) -- (2.5,2) -- (3,0) -- cycle;
\draw (5,0) -- (5.25,2) -- (5.5,3) -- (6,4) -- (6.25,3) -- (6.5,2) -- (7,0) -- cycle;
\fill [gray,opacity=0.3] (5,0) -- (5.25,2) -- (5.5,3) -- (6,4) -- (6.25,3) -- (6.5,2) -- (7,0) -- cycle;

\draw [dashed] (-0.1,2) -- (6.5,2);
\draw [dashed] (-0.1,3) -- (6.27,3);
\draw [dashed] (-0.1,4) -- (6,4);

\draw (0,1) node [left] {\small $I_1^+$};
\draw (0,2.5) node [left] {\small $I_2^+$};
\draw (0,3.5) node [left] {\small $I_3^+$};

\draw (0.75,1) node {\small $u_1^1$};
\draw (3.25,1) node {\small $v_1^1$};

\draw (4.6,1) node {\small $u_1^2$};
\draw (7.25,1) node {\small $v_1^2$};

\draw (5.2,3.5) node {\small $u_3^1$};
\draw (6.6,3.5) node {\small $v_3^1$};

\draw (1.1,2.5) node {\small $u_2^1$};
\draw (2.75,2.5) node {\small $v_2^1$};

\draw (4.8,2.5) node {\small $u_2^2$};
\draw (6.8,2.5) node {\small $v_2^2$};

\draw (2,0.5) node {\small $E$};
\draw (6,0.5) node {\small $E$};

\begin{scope} [shift={(10,0)}]

\draw [thick,->] (-0.5,0) -- (6,0); 
\draw [thick,->] (0,-0.5) -- (0,4.5); 
\draw (6,0) node [right] {\small $x$};
\draw (0,4.5) node [above] {\small $y$};

\draw (1,0) -- (1.5,2) -- (2,3) -- (2.5,4) -- (2.8,3) -- (3.8,2) -- (5,0) -- cycle;
\fill [gray,opacity=0.3] (1,0) -- (1.5,2) -- (2,3) -- (2.5,4) -- (2.8,3) -- (3.8,2) -- (5,0) -- cycle;

\draw [dashed] (-0.1,2) -- (3.8,2);
\draw [dashed] (-0.1,3) -- (2.8,3);
\draw [dashed] (-0.1,4) -- (2.5,4);

\draw (1.1,2.5) node {\small $\tilde{u}$};
\draw (3.75,2.5) node {\small $\tilde{v}$};

\draw (3,0.5) node {\small $E^\bigstar$};

\end{scope}

\end{tikzpicture}
\caption{If $E$ is a bounded open set with polyhedral boundary, with outer unit normal never orthogonal to $\mathrm{e}_1$, then $\de E$ is parametrized as in~\eqref{eq:boundary_polyhed}. In particular, $u_h^k=v_h^k$ on $\de I\cap\de I_h$. Moreover, if $\de I_i\cap \de I_h\ne\varnothing$ and $1\le k\le N(h)$, then either $u_h^k =v_h^k$, or there exists $1\le j\le N(i)$ such that $u_h^k=u_i^j$ and $v_h^k=v_i^j$ on $\de I_i\cap\de I_h$. These two properties, guaranteed by~\eqref{eq:boundary_polyhed}, imply the continuity of~$\lambda_E$, $\tilde{u}$ and~$\tilde{v}$.}
\label{fig:rearrang_polyhed}
\end{figure}

%% file: trace_increase.tex
\begin{figure}
\begin{tikzpicture}[scale=.8]
\draw [thick,->] (-0.5,0) -- (6,0); 
\draw [thick,->] (0,-0.5) -- (0,4.5); 
\draw (6,0) node [right] {\small $x$};
\draw (0,4.5) node [above] {\small $y$};
\draw (0,0) -- (0,1) -- (1,0) -- cycle;
\fill [gray,opacity=0.3] (0,0) -- (0,1) -- (1,0) -- cycle;
\draw (1,0) -- (1,2) -- (3,0) -- cycle;
\fill [gray,opacity=0.3] (1,0) -- (1,2) -- (3,0) -- cycle;
\draw (3,0) -- (3,2) -- (4,3) -- (4,4) -- (5,0) -- cycle;
\fill [gray,opacity=0.3] (3,0) -- (3,2) -- (4,3) -- (4,4) -- (5,0) -- cycle;
\draw [dashed] (-0.1,1) -- (4.8,1);
\draw [dashed] (-0.1,2) -- (4.6,2);
\draw [dashed] (-0.1,3) -- (4.3,3);
\draw [dashed] (-0.1,4) -- (4,4);
\draw (4,0.5) node {\small $E$};
\draw [ultra thick] (0,0) -- (0,1);
\draw (-1.05,0.5) node {\tiny $\de E\cap\de H$};
\draw [decorate,decoration={brace,amplitude=3,raise=2}](0,0.1) -- (0,0.9);
\begin{scope} [shift={(9,0)}]
	\draw [thick,->] (-0.5,0) -- (6,0); 
	\draw [thick,->] (0,-0.5) -- (0,4.5); 
	\draw (6,0) node [right] {\small $x$};
	\draw (0,4.5) node [above] {\small $y$};
	\draw (0,0) -- (0,2) -- (1,3) -- (1,4) -- (1.6,2) -- (3.8,1) -- (5,0) -- cycle;
	\fill [gray,opacity=0.3] (0,0) -- (0,2) -- (1,3) -- (1,4) -- (1.6,2) -- (3.8,1) -- (5,0) -- cycle;
	\draw [dashed] (-0.1,1) -- (3.8,1);
	\draw [dashed] (-0.1,2) -- (1.6,2);
	\draw [dashed] (-0.1,3) -- (1.3,3);
	\draw [dashed] (-0.1,4) -- (1,4);
	\draw (2,0.5) node {\small $E^\bigstar$};
	\draw [ultra thick] (0,0) -- (0,2);
	\draw (-1.3,1) node {\tiny $\de E^\bigstar\cap\de H$};
	\draw [decorate,decoration={brace,amplitude=5,raise=2}](0,0.1) -- (0,1.9);
\end{scope}
\end{tikzpicture}
\caption{In general, the amount of perimeter on $\de H$ of a bounded open set $E$ with polyhedral boundary is increased by the rearrangement defined in~\eqref{eq:def_rearrangement}, so that $P(E;\de H)\le P(E^\bigstar;\de H)$. Precisely, the set $\de E^\bigstar\cap\de H$ coincides with the connected component of~$\set{\phi_E'=0}$ containing $\de E\cap\de H$.}
\label{fig:trace_increase}
\end{figure}
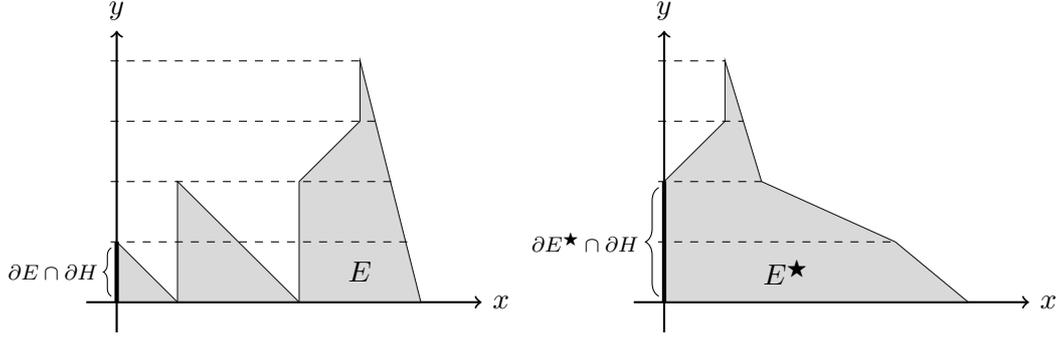

%% file: dbeucl.tex
\begin{figure}[h!]
\pgfmathsetmacro{\s}{0.9} 
\tikzset{
    mdbeuclid/.pic={
\draw [blue,thick,domain=-117:117] plot ({1+sqrt(5)*cos(\x)}, {sqrt(5)*sin(\x)});
\fill [blue,opacity=0.1,domain=-117:117] plot ({1+sqrt(5)*cos(\x)}, {sqrt(5)*sin(\x)}); 
}}
\begin{tikzpicture}[scale=\s]
\draw [thick,->] (-4,0) -- (4,0);
\draw [thick,->] (0,-3) -- (0,3);
\draw (4,0) node [right] {\small $x$};
\draw (0,3) node [above] {\small $y$};
\draw (1,0) node [below] {\tiny $\frac{1}{2|k|}$} node {\tiny\textbullet};
\draw (3.6,0) node [below] {\tiny $\frac{3}{2|k|}$};
\draw [blue,thick,dashed] (1,0) -- (0,2);
\pic [scale=\s] {mdbeuclid};
\pic [scale=\s,xscale=-1] {mdbeuclid};
\draw (2.25,2.2) node [right] {\small $f$};
\draw (0,0.8) node [below,right] {\small $\theta$};
\draw [domain=270:300] plot ({0.9*cos(\x)}, {2+0.9*sin(\x)});
\draw [red,thick,shift={(0,1.6)},rotate around={30:(0,0.4)}] (0,0) -- (0.4,0) -- (0.4,0.4);
\draw [red,thick,shift={(0,1.6)},rotate around={30:(0,0.4)}] (0,0.4) -- (1.25,0.4);
\draw (0.5,1.8) node [below,right] {$\frac{\pi}{2}$};
\end{tikzpicture}
\caption{The minimizer of problem~\eqref{eq:minpart} in the Euclidean case $\alpha=0$.}
\label{fig:dbeucl}
\end{figure}
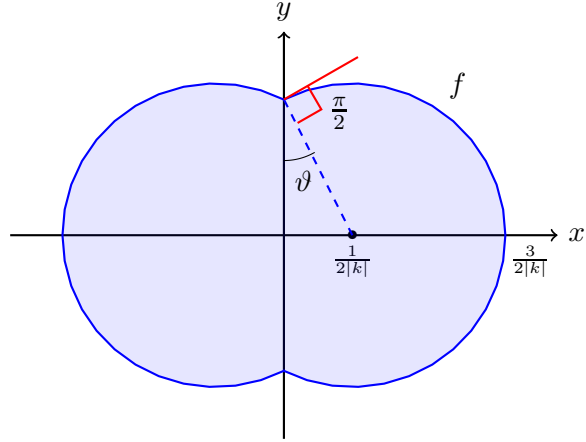

%% file: grushin_bubble.tex
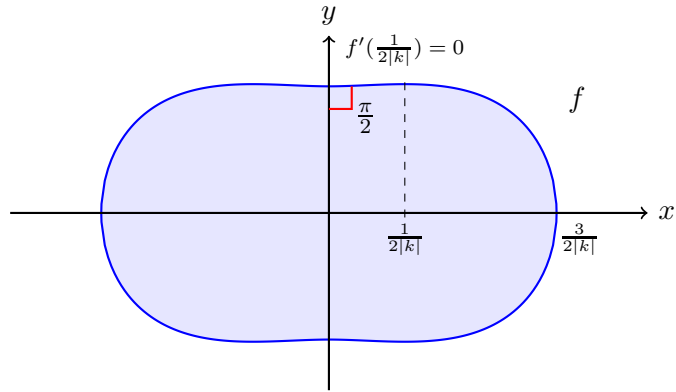
\begin{figure}[h!]
\pgfmathsetmacro{\k}{-exp(ln((8*pi+9*sqrt(3))/(12))/3)}
\pgfmathsetmacro{\d}{0.5}
\pgfmathsetmacro{\dom}{(-1-\d)/\k}
\pgfmathsetmacro{\xdom}{1.4*\dom}
\pgfmathsetmacro{\codom}{(pi/2+asin(\d)/180*pi+(\d)*sqrt(1-(\d)^2))/(2*(\k)^2)}
\pgfmathsetmacro{\ydom}{1.4*\codom}
\pgfmathsetmacro{\s}{70}
\pgfmathsetmacro{\ss}{3}
\tikzset{
    profile/.pic={
\draw [thick,blue] 
 plot 
      [domain=0:\dom, samples=\s]  
(\x,{(pi/2+asin(\k*\x+\d)/180*pi+(\d-\k*\x)*sqrt(1-(\k*\x+\d)^2))/(2*(\k)^2)}) -- (\dom,0);
\fill [opacity=0.1,blue] 
 plot 
      [domain=0:\dom, samples=\s]  
(\x,{(pi/2+asin(\k*\x+\d)/180*pi+(\d-\k*\x)*sqrt(1-(\k*\x+\d)^2))/(2*(\k)^2)}) -- (\dom,0) -- (0,0) -- cycle;
    }}
\begin{tikzpicture}[scale=\ss]
\pic [scale=\ss] {profile};
\pic [scale=\ss,yscale=-1] {profile};
\pic [scale=\ss,xscale=-1] {profile};
\pic [scale=\ss,xscale=-1,yscale=-1] {profile};
\draw [thick,->] (-\xdom,0) -- (\xdom,0);
\draw [thick,->] (0,-\ydom) -- (0,\ydom);
\draw (\xdom,0) node [right] {\small $x$};
\draw (\dom+.1,0) node [below] {\tiny $\frac{3}{2|k|}$};
\draw (0,\ydom) node [above] {\small $y$};
\draw (1,0.5) node [right] {\small $f$};
\draw (-\d/\k,0) node [below] {\tiny $\frac{1}{2|k|}$};
\draw [dashed,thin] (-\d/\k,-0.02) -- (-\d/\k,0.6);
\draw (-\d/\k,0.6) node [above] {\tiny $f'(\tfrac{1}{2|k|})=0$};
\draw [thick,red] (0.1,\codom) -- (0.1,\codom-0.1) -- (0,\codom-0.1);
\draw (0.16,\codom-0.14) node {\small $\frac{\pi}{2}$};
\end{tikzpicture}
\caption{The minimizer of problem~\eqref{eq:minpart} in the Grushin case $\alpha=1$.}
\label{fig:grushin_bubble}
\end{figure}

%% file: angles.tex
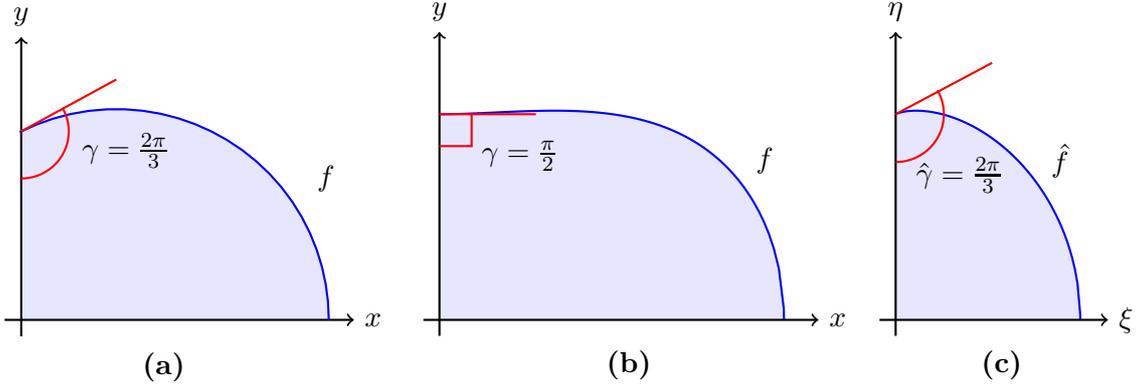
\begin{figure}[h!]
\pgfmathsetmacro{\k}{-exp(ln((16*pi+9*sqrt(3))/(24))/3)}
\pgfmathsetmacro{\d}{0.5}
\pgfmathsetmacro{\dom}{(-1-\d)/\k}
\pgfmathsetmacro{\xdom}{1.1*\dom}
\pgfmathsetmacro{\domt}{((\dom)^2)/2}
\pgfmathsetmacro{\xdomt}{1.15*\domt}
\pgfmathsetmacro{\codom}{(pi/2+asin(\d)/180*pi+(\d)*sqrt(1-(\d)^2))/(2*(\k)^2)}
\pgfmathsetmacro{\ydom}{1.4*\codom}
\pgfmathsetmacro{\s}{70}
\pgfmathsetmacro{\ss}{4.23}
\tikzset{
    profile/.pic={
\draw [thick,blue] 
 plot 
      [domain=0:\dom, samples=\s]  
(\x,{(pi/2+asin(\k*\x+\d)/180*pi+(\d-\k*\x)*sqrt(1-(\k*\x+\d)^2))/(2*(\k)^2)}) -- (\dom,0);
\fill [opacity=0.1,blue] 
 plot 
      [domain=0:\dom, samples=\s]  
(\x,{(pi/2+asin(\k*\x+\d)/180*pi+(\d-\k*\x)*sqrt(1-(\k*\x+\d)^2))/(2*(\k)^2)}) -- (\dom,0) -- (0,0) -- cycle;
    }}
\tikzset{
    profilet/.pic={
\draw [thick,blue] 
 plot 
      [domain=0:\domt, samples=\s]  
(\x,{(pi/2+asin(\k*((2*\x)^0.5)+\d)/180*pi+(\d-\k*((2*\x)^0.5))*sqrt(1-(\k*((2*\x)^0.5)+\d)^2))/(2*(\k)^2)}) -- (\domt,0);
\fill [opacity=0.1,blue] 
 plot 
      [domain=0:\domt, samples=\s]  
(\x,{(pi/2+asin(\k*((2*\x)^0.5)+\d)/180*pi+(\d-\k*((2*\x)^0.5))*sqrt(1-(\k*((2*\x)^0.5)+\d)^2))/(2*(\k)^2)}) -- (\domt,0) -- (0,0) -- cycle;
    }}
\begin{tikzpicture}
\begin{scope}[scale=1.25]
	\draw [thick,->] (-0.175,0) -- (3.5,0);
	\draw [thick,->] (0,-0.175) -- (0,3);
	\draw (3.5,0) node [right] {\small $x$};
	\draw (0,3) node [above] {\small $y$};
	\draw [blue,thick,domain=0:117] plot ({1+sqrt(5)*cos(\x)}, {sqrt(5)*sin(\x)});
	\fill [blue,opacity=0.1,domain=0:117] plot ({1+sqrt(5)*cos(\x)}, {sqrt(5)*sin(\x)})--(0,0)--cycle;
	\draw (3,1.5) node [right] {\small $f$};
	\draw [domain=30:-90,,thick,red] plot ({0.5*cos(\x)}, {2+0.5*sin(\x)});
	\draw [thick,red] (0,2) -- (1,2.55);
	\draw (1.1,1.8) node {\small $\gamma=\frac{2\pi}{3}$};
	\draw (1.5,-.5) node {\small\bf (a)};
\end{scope}
\begin{scope}[shift={(5.5,0)},scale=\ss]
	\pic [scale=\ss] {profile};
	\draw [thick,->] (-0.05,0) -- (\xdom,0);
	\draw [thick,->] (0,-0.05) -- (0,\ydom);
	\draw (\xdom,0) node [right] {\small $x$};
	\draw (0,\ydom) node [above] {\small $y$};
	\draw [thick,red] (0,\codom) -- (.3,\codom);
	\draw [thick,red] (.1,\codom) -- (.1,\codom-.1) -- (0,\codom-.1);
	\draw (.25,\codom-.13) node {\small $\gamma=\frac{\pi}{2}$};
	\draw (.95,0.5) node [right] {\small $f$};
	\draw (\xdom/2,-.14) node {\small\bf (b)};
\end{scope}
\begin{scope}[shift={(11.5,0)},scale=\ss]
	\pic [scale=\ss] {profilet};
	\draw [thick,->] (-0.05,0) -- (\xdomt,0);
	\draw [thick,->] (0,-0.05) -- (0,\ydom);
	\draw (\xdomt,0) node [right] {\small $\xi$};
	\draw (0,\ydom) node [above] {\small $\eta$};
	\draw (0.45,0.5) node [right] {\small $\hat{f}$};
	\draw [domain=30:-90,thick,red] plot ({0.15*cos(\x)}, {\codom+0.15*sin(\x)});
	\draw [thick,red] (0,\codom) -- (.3,\codom*2/1.6);
	\draw (.2,\codom-.19) node {\small $\hat{\gamma}=\frac{2\pi}{3}$};
	\draw (\xdomt/2,-.14) node {\small\bf (c)};
\end{scope}
\end{tikzpicture}
\caption{The profile angle at the interface in problem~\eqref{eq:minpart}: \textbf{(a)} the Euclidean case $\alpha=0$; \textbf{(b)} the Grushin case $\alpha>0$; \textbf{(c)} the transformed Grushin case $\alpha>0$.}
\label{fig:angles}
\end{figure}

%% file: dbeucl_y.tex
\begin{figure}[h!]
\pgfmathsetmacro{\s}{0.7} 
\tikzset{
    mdbeuclid_y/.pic={
\draw [blue,thick,domain=-27:207] plot ({sqrt(5)*cos(\x)}, {1+sqrt(5)*sin(\x)});
\fill [blue,opacity=0.1,domain=-27:207] plot ({sqrt(5)*cos(\x)}, {1+sqrt(5)*sin(\x)});    }}
\begin{tikzpicture}[scale=\s]
\draw [thick,->] (-2.5,0) -- (2.5,0);
\draw [thick,->] (0,-3.75) -- (0,3.75);
\draw (2.5,0) node [right] {\small $x$};
\draw (0,3.75) node [above] {\small $y$};
\draw (0,1) node [left] {\tiny $\frac{1}{2k}$} node {\tiny\textbullet};
\draw (0,3.6) node [left] {\tiny $\frac{3}{2k}$};
\draw [blue,thick,dashed] (0,1) -- (2,0);

\pic [scale=\s] {mdbeuclid_y};
\pic [scale=\s,yscale=-1] {mdbeuclid_y};

\draw (2.25,2.2) node [right] {\small $g$};
\draw (0,0.4) node [below,right] {\small $\theta$};
\draw [domain=150:180] plot ({2+0.9*cos(\x)}, {0.9*sin(\x)});
\draw [red,thick,shift={(1.6,0)},rotate around={-30:(0.4,0)}] (0,0) -- (0,0.4) -- (0.4,0.4);
\draw [red,thick,shift={(1.6,0)},rotate around={-30:(0.4,0)}] (0.4,0) -- (0.4,1.25);
\draw (1.5,1) node [below,right] {$\frac{\pi}{2}$};

\end{tikzpicture}
\caption{The minimizer of problem~\eqref{eq:minpart_y} in the Euclidean case $\alpha=0$.}
\label{fig:dbeucl_y}
\end{figure}
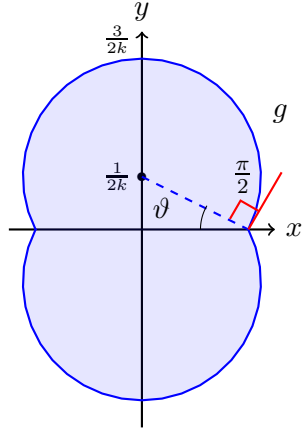

%% file: isop_bubble.tex
\begin{figure}[h!]
\pgfmathsetmacro{\s}{70}
\pgfmathsetmacro{\ss}{1.75}
\pgfmathsetmacro{\t}{exp(0.5*ln(3))/2} 
\pgfmathsetmacro{\l}{exp(ln(12/17)/3)} 
\pgfmathsetmacro{\pt}{(0.5)*(acos(\t)/180*pi+(\t)*sqrt(1-(\t)^2))} 

\tikzset{
    profile/.pic={
\draw [thick,blue] 
 plot 
      [domain=0:1, samples=\s]  
(\x,{(0.5)*(acos(\x)/180*pi+(\x)*sqrt(1-(\x)^2))+\pt}) -- (1,\pt);
\fill [opacity=0.1,blue] 
 plot 
      [domain=0:1, samples=\s]  
(\x,{(0.5)*(acos(\x)/180*pi+(\x)*sqrt(1-(\x)^2))+\pt}) -- (1,\pt) -- (0,\pt) -- cycle;
    }}

\tikzset{
    tau_profile/.pic={
\draw [thick,blue] 
 plot 
      [domain=\t:1, samples=\s]  
(\x,{-(0.5)*(acos(\x)/180*pi+(\x)*sqrt(1-(\x)^2))+\pt}) -- (1,\pt);
\fill [opacity=0.1,blue] 
 plot 
      [domain=\t:1, samples=\s]  
(\x,{-(0.5)*(acos(\x)/180*pi+(\x)*sqrt(1-(\x)^2))+\pt}) -- (1,\pt) -- (0,\pt) -- (0,0) -- cycle;
    }}

\tikzset{
    bubble_top/.pic={
\begin{scope}
\pic [scale=\ss,xscale=\l,yscale=\l*\l*\l]{profile};
\pic [scale=\ss,xscale=-\l,yscale=\l*\l*\l] {profile};
\pic [scale=\ss,xscale=\l,yscale=\l*\l*\l] {tau_profile};
\pic [scale=\ss,xscale=-\l,yscale=\l*\l*\l] {tau_profile};
\end{scope}
    }}
    
\tikzset{
    bubble_bottom/.pic={
\begin{scope}
\pic [scale=\ss,xscale=\l,yscale=-\l*\l*\l] {profile};
\pic [scale=\ss,xscale=-\l,yscale=-\l*\l*\l] {profile};
\pic [scale=\ss,xscale=\l,yscale=-\l*\l*\l] {tau_profile};
\pic [scale=\ss,xscale=-\l,yscale=-\l*\l*\l] {tau_profile};
\end{scope}
    }}
    
\begin{tikzpicture}[scale=\ss]
\pic {bubble_top};
\pic {bubble_bottom};
\draw [thick,->] (-1.25,0) -- (1.25,0);
\draw [thick,->] (0,-1.25) -- (0,1.25);
\draw (0,2.5*\l*\l*\pt) node [left] {\tiny $r$};
\draw (1.25,0) node [right] {\small $x$};
\draw (0,1.25) node [above] {\small $y$};
\draw (0,\l*\l*\pt) node [left] {\tiny $\hat{y}$};
\draw [dashed,thin] (0,\l*\l*\pt) -- (\l,\l*\l*\pt);
\draw (\l,\l*\l*\pt) node [right] {\tiny $g'(\hat{y})=0$};
\draw (0.8*\l,2*\l*\l*\pt) node [right] {\ $g$};
\end{tikzpicture}
\caption{The minimizer of problem~\eqref{eq:minpart_y} in the Grushin case $\alpha=1$.}
\label{fig:isop_bubble}
\end{figure}
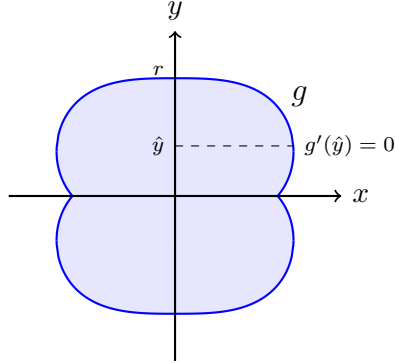

%% file: angles_y.tex
\begin{figure}[h!]

\pgfmathsetmacro{\s}{70}
\pgfmathsetmacro{\ss}{2.5}

\pgfmathsetmacro{\t}{exp(0.5*ln(3))/2} 
\pgfmathsetmacro{\l}{exp(ln(12/17)/3)} 
\pgfmathsetmacro{\pt}{(0.5)*(acos(\t)/180*pi+(\t)*sqrt(1-(\t)^2))} 

\tikzset{
    euclid_y/.pic={
\draw [blue,thick,domain=-27:90] plot ({sqrt(5)*cos(\x)}, {1+sqrt(5)*sin(\x)});
\fill [blue,opacity=0.1,domain=-27:90] plot ({sqrt(5)*cos(\x)}, {1+sqrt(5)*sin(\x)}) -- (0,0) -- cycle;    
}}

\tikzset{
    profile/.pic={
\draw [thick,blue] 
 plot 
      [domain=0:1, samples=\s]  
(\x,{(0.5)*(acos(\x)/180*pi+(\x)*sqrt(1-(\x)^2))+\pt}) -- (1,\pt);
\fill [opacity=0.1,blue] 
 plot 
      [domain=0:1, samples=\s]  
(\x,{(0.5)*(acos(\x)/180*pi+(\x)*sqrt(1-(\x)^2))+\pt}) -- (1,\pt) -- (0,\pt) -- cycle;
    }}

\tikzset{
    tau_profile/.pic={
\draw [thick,blue] 
 plot 
      [domain=\t:1, samples=\s]  
(\x,{-(0.5)*(acos(\x)/180*pi+(\x)*sqrt(1-(\x)^2))+\pt}) -- (1,\pt);
\fill [opacity=0.1,blue] 
 plot 
      [domain=\t:1, samples=\s]  
(\x,{-(0.5)*(acos(\x)/180*pi+(\x)*sqrt(1-(\x)^2))+\pt}) -- (1,\pt) -- (0,\pt) -- (0,0) -- cycle;
    }}

\tikzset{
    half_bubble_top/.pic={
\begin{scope}
\pic [scale=\ss,xscale=\l,yscale=\l*\l*\l]{profile};
\pic [scale=\ss,xscale=\l,yscale=\l*\l*\l] {tau_profile};
\end{scope}
    }}
    
\tikzset{
    profile_t/.pic={
\draw [thick,blue] 
 plot 
      [domain=0:1, samples=\s]  
({(\x*\x)/2},{(0.5)*(acos(\x)/180*pi+(\x)*sqrt(1-(\x)^2))+\pt}) -- (1,\pt);
\fill [opacity=0.1,blue] 
 plot 
      [domain=0:1, samples=\s]  
({(\x*\x)/2},{(0.5)*(acos(\x)/180*pi+(\x)*sqrt(1-(\x)^2))+\pt}) -- (1,\pt) -- (0,\pt) -- cycle;
    }}
    
\tikzset{
    tau_profile_t/.pic={
\draw [thick,blue] 
 plot 
      [domain=\t:1, samples=\s]  
({1/2+\x*\x/2},{-(0.5)*(acos(\x)/180*pi+(\x)*sqrt(1-(\x)^2))+\pt}) -- (1,\pt);
\fill [opacity=0.1,blue] 
 plot 
      [domain=\t:1, samples=\s]  
({1/2+\x*\x/2},{-(0.5)*(acos(\x)/180*pi+(\x)*sqrt(1-(\x)^2))+\pt}) -- (1,\pt) -- (0,\pt) -- (0,0) -- cycle;
    }}
    
\tikzset{
    half_bubble_top_t/.pic={
\begin{scope}
\pic [scale=\ss,xscale=\l,yscale=\l*\l*\l] {profile};
\pic [scale=\ss,xscale=\l,yscale=\l*\l*\l] {tau_profile_t};
\end{scope}
    }}

\begin{tikzpicture}
\begin{scope}[scale=.8]
	\pic [scale=.8] {euclid_y};
	\draw [thick,->] (-.5,0) -- (3,0);
	\draw [thick,->] (0,-.5) -- (0,4);
	\draw (3,0) node [right] {\small $x$};
	\draw (0,4) node [above] {\small $y$};
	\draw (2.4,1.5) node [right] {\small $g$};
	\draw [domain=180:60,,thick,red] plot ({2+0.5*cos(\x)}, {0.5*sin(\x)});
	\draw [thick,red] (2,0) -- (2.55,1);
	\draw (1.3,1.1) node {\small $\gamma=\frac{2\pi}{3}$};
	\draw (1.5,-.75) node {\small\bf (a)};
\end{scope}
\begin{scope}[shift={(4.5,0)},scale=\ss]
	\pic [scale=\ss] {half_bubble_top};
	\draw [thick,->] (-.15,0) -- (1.25,0);
	\draw [thick,->] (0,-.15) -- (0,1.25);
	\draw (1.25,0) node [right] {\small $x$};
	\draw (0,1.25) node [above] {\small $y$};
	\draw [domain=180:45,,thick,red] plot ({\l*\t+0.16*cos(\x)}, {0.16*sin(\x)});
	\draw [thick,red] (\l*\t,0) -- (1.07,.3);
	\draw (.4,.3) node {\small $\gamma=\frac{\pi}{2}+\vartheta$};
	\draw (.95,0.5) node [right] {\small $g$};
	\draw (1.25/2,-.25) node {\small\bf (b)};
\end{scope}
\begin{scope}[shift={(9.75,0)},scale=\ss]
	\pic [scale=\ss] {half_bubble_top_t};
	\draw [thick,->] (-.15,0) -- (1.25,0);
	\draw [thick,->] (0,-.15) -- (0,1.25);
	\draw (1.25,0) node [right] {\small $x$};
	\draw (0,1.25) node [above] {\small $y$};
	\draw [domain=180:60,,thick,red] plot ({0.8+0.16*cos(\x)}, {0.16*sin(\x)});
	\draw (.6,.3) node {\small $\gamma=\frac{2\pi}{3}$};
	\draw [thick,red] (.78,0) -- (1,.3);
	\draw (1.05,0.5) node {\small $\hat{g}$};
	\draw (1.25/2,-.25) node {\small\bf (c)};
\end{scope}
\end{tikzpicture}
\caption{The profile angle at the interface in problem~\eqref{eq:minpart_y}: \textbf{(a)} the Euclidean case $\alpha=0$; \textbf{(b)} the Grushin case $\alpha>0$; \textbf{(c)} the transformed Grushin case $\alpha>0$.}
\label{fig:angles_y}
\end{figure}
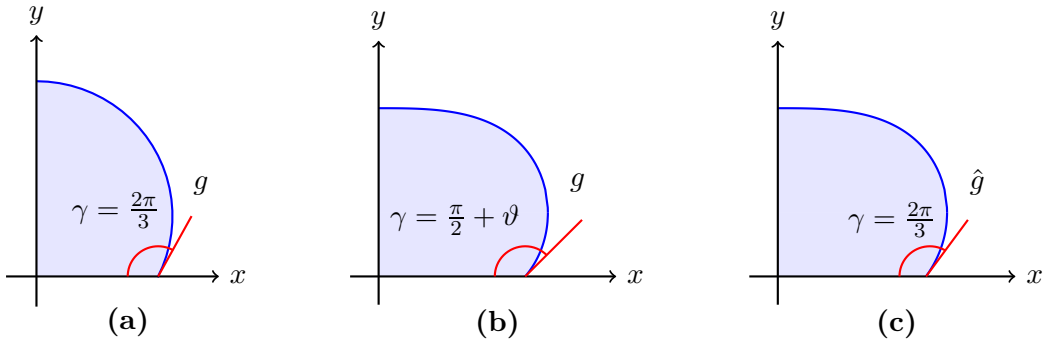